\documentclass[final,leqno]{siamart171218}
\usepackage{array}
\usepackage{arydshln}
\usepackage{multirow}
\usepackage{amssymb,amsmath}
\usepackage{amscd}
\usepackage{amsfonts}
\usepackage{enumerate}
\usepackage{xypic}
\usepackage{stmaryrd}
\usepackage{graphicx}
\usepackage{fancybox}
\usepackage{color}
\usepackage{amsmath}
\usepackage{exscale}
\usepackage{multirow}
\usepackage{subfigure}
\usepackage{enumitem,array,algorithm,algpseudocode}
\usepackage{graphicx}
\usepackage{mathtools}

\makeatletter

\newlist{myitemize}{itemize}
{2}
\setlist[myitemize,1]{label={$\bullet$},leftmargin=1cm}
\newtheorem{preremark}{Remark}[section]
\newenvironment{remark}{\vspace{1ex}\begin{preremark}\hspace{0.2mm}\rm}{\end{preremark}\vspace{1ex}}
\newtheorem{preprop}{Proposition}[section]

\newcounter{example}

\begin{document}

\title{Solving the 3D high-frequency Helmholtz equation
using contour integration and polynomial preconditioning}

\author{Xiao Liu \thanks{Department of Computational and Applied Mathematics, Rice
University, Houston, TX 77005 (xiao.liu@rice.{\allowbreak}edu)}
\and Yuanzhe Xi \thanks{Department of Mathematics, Emory University, GA 30322 (yxi26@emory.{\allowbreak}edu)}
\and Yousef Saad \thanks{Department of Computer Science and Engineering,
University of Minnesota, Minneapolis, MN 55455
(saad@umn.{\allowbreak}edu) }
\and Maarten V. de Hoop
\thanks{Department of Computational and Applied mathematics,
Rice University, Houston, TX 77005
(mdehoop@rice.{\allowbreak}edu)}}

\maketitle

\begin{abstract}
  We propose an iterative solution method for the 3D high-frequency Helmholtz
  equation that exploits  a contour integral formulation of spectral
  projectors. In this  framework, the solution in certain invariant subspaces 
is approximated by solving complex-shifted linear systems,
resulting in faster  GMRES iterations due to the  restricted  spectrum.
 The shifted systems are solved by
 exploiting a polynomial fixed-point iteration, 
which is a robust scheme even if the magnitude of the shift is small.
Numerical tests in 3D indicate  that $O(n^{1/3})$ matrix-vector products
are needed to solve a high-frequency problem with a matrix size $n$ with
high accuracy. The method has a small storage requirement,
can be applied to both dense and sparse linear systems,
and is highly parallelizable. 
\end{abstract}



\pagestyle{myheadings} \thispagestyle{plain}
\markboth{X. LIU, Y. XI, Y. SAAD, AND M. V. DE HOOP}
{\uppercase{Matrix-free Helmholtz solution}}

\section{Introduction}
\subsection{Problem of interest}
Helmholtz-type equations are second-order partial differential equations 
that model time-harmonic waves in materials with linear constitutive relations.
For the scalar case, the Helmholtz operator can be written as
\begin{equation}
-\Delta-\omega^2/c^2(x),
\label{eq:helmholtz-op}
\end{equation}
where $\omega$ is the angular frequency,
and $c(x)$ is the wavespeed.
After one of several  types of discretizations is applied,  
we end up with  an $n\times n$ linear system of the following form to solve:
\begin{equation}
Au=f.
\label{eq:lin-sys}
\end{equation}
The main subject of this paper is the fast iterative solution
of \eqref{eq:lin-sys}.
The linear system \eqref{eq:lin-sys} is challenging to solve
because the coefficient matrix $A$ is typically \emph{highly indefinite} and
\emph{non-Hermitian}.

\subsection{Existing work}
Iterative methods  can show fast convergence  for non-Hermitian linear
systems,  when the  spectrum of  the coefficient  matrix lies  in some
confined  region  in  the  complex plane  that  excludes  the  origin.
Chebyshev                   iteration                   \cite{wrigley,
  manteuffel77,manteuffel78,gutknecht}  is  among  the first  type  of
modern iterative  methods developed for  the non-Hermitian case.  It is
suitable for well-conditioned cases where  the spectrum is enclosed in
an ellipse  some distance  away from  the origin.  Later, alternatives
using least-squares  polynomials \cite{saad-lspoly}, were  designed by
minimizing some  weighted L2 norm  of the  residual on a  polygon that
encloses the spectrum. At the  same time Krylov subspace methods, such
as  GMRES \cite{saad-gmres},  were  found appealing  as  they did  not
require any prior spectral information. These methods converge well if
the  numerical  range   is  not  close  to   the  origin  \cite{elman,
  beckermann}.  For (modified)  Hermitian and skew-Hermitian splitting
methods \cite{bai-hss, bai-mhss}, the spectrum is assumed to be in one
of the  four quadrants. However,  the performance of these  methods in
solving  \eqref{eq:lin-sys} is  usually unacceptably  poor due  to the
unfavorable spectrum of the discretized Helmholtz operators.

In order to obtain  iterative schemes that converge fast when
solving \eqref{eq:lin-sys},
several efficient preconditioners such as optimized Schwarz method \cite{benamou,gander-nodd,gander-dd},
PML sweeping preconditioners \cite{enquist-pml, stolk-dd, zepeda},
and shifted Laplacian preconditioners \cite{bayliss,erlangga,vangijzen07} have recently  been proposed. These preconditioners are much more expensive to
construct than those for solving standard elliptic PDEs. We highlight
specifically the idea  of shifted Laplacian preconditioners
\cite{erlangga} which relies on
the fact  it is easier to solve linear systems with the shifted matrix
$A-zI$, for some complex number $z$, than with the original matrix $A$.
The idea of complex shifts is generalized in \cite{xi-ration}
based on contour integration formulations
that are formerly used in eigenvalue computations \cite{feast,ls-feast}.
The method draws a circular contour in the complex plane to decompose the spectrum of $A^{-1}$,
and the decomposed sub-problems are then solved separately.

The application  of shifted Laplacian--type preconditioners  $A-zI$ is
usually  based  on  extending  standard elliptic  solvers  to  complex
matrices. Two popular choices are incomplete LU (ILU) \cite{saad-ilut,
  osei-shift} and  multigrid methods \cite{mac-camg}. One  major issue
associated  with  these  preconditioners  is  that  their  performance
deteriorates   dramatically   as   the  angular   frequency   $\omega$
increases. This is  because on the one hand, a  small magnitude of $z$
is  necessary for  convergence  when  solving high-frequency  problems
\cite{gander-smallshift} and on  the  other  hand, a  small  $|z|$  will
significantly  increase   the  computational  burden   to  approximate
$(A-zI)^{-1}$:  standard multigrid  methods are  no longer  guaranteed
to be effective \cite{cocquet-largeshift}, and ILU  factors become dense and
even  unstable.  The aim of this paper
is  to  propose  an  efficient  and  robust
preconditioning technique to overcome these difficulties.

\subsection{Outline of the proposed method}
In this paper, we solve the 3D Helmholtz problems
in a contour integration framework adapted from \cite{xi-ration},
with a new fixed-point iteration for the shifted systems.
The fixed-point iteration is based on a polynomial approximation of the
matrix exponential,
which is suitable for the case when the spectrum is  
confined in a rectangle with a small separation away from the origin and
the standard Chebyshev iteration diverges. Compared with existing methods
for solving shifted problems,
the new approach is robust and has a fixed storage requirement
even if the imaginary part of the shift nears the origin.
In the proposed  contour integration framework,
the fixed-point iteration is used to resolve components of  
the sub-problem associated  with large eigenvalues of $A$,
and GMRES is used to resolve the remainder.

For the Helmholtz  equation with the impedance  boundary condition, we
show for an idealized case that inside some contour the imaginary part
of  each   eigenvalue  is
well separated   from  the  real   axis.   Then,
$\mathrm{i}A$  may have  a positive  definite Hermitian  part for  the
problem  inside  the  contour,  in   which  case  the  Elman  estimate
\cite{elman} gives a  rather good estimation of  the GMRES convergence
rate.  We give several techniques  to improve the effectiveness of the
solver and  demonstrate the performance of the proposed scheme
for challenging high-frequency
variable-coefficient problems  in 3D. It  is well known  that spectral
methods require fewer grid points  per wavelength relative to finite
difference and finite element methods \cite{shenH}. Since the proposed
method only accesses the matrix  through matrix-vector products, it is
ideally  suited  for  solving  dense linear  systems  resulted  from
spectral discretizations.

The remaining  sections are  organized as follows.   In Section  2, we
review the contour integration framework for general indefinite linear
systems.  In Section 3, we  characterize the spectrum of the Helmholtz
problem based on linear  algebra assumptions.  A fixed-point iteration
method  is developed  in  Section  4 to  solve  shifted problems.   In
Section  5,  convergence  is  studied based  on  the  distribution  of
eigenvalues of the interior impedance problem.  Some useful techniques
are provided  in Section 6  to achieve  an optimal performance  of the
proposed method.  In Section 7, numerical examples are presented using
Fourier spectral and finite difference methods.  Conclusions are drawn
in Section 8.

The following notation will be used throughout the remaining sections:
\begin{myitemize}
\item $\operatorname{Range}(G)$ and $\operatorname{Null}(G)$ denote the range and null space, respectively, of a matrix $G$; 
\item $\rho(G)$ represents the spectral radius of a matrix $G$;
\item 
$G\succ(\succeq)\ 0$ means $G$ is Hermitian positive (semi-)definite.
\end{myitemize}
\section{Review of the contour integration framework}\label{sec:fci}
The inverse of a matrix $A$ can be approximated by a linear combination of
the resolvent $(A-zI)^{-1}$ with several complex shifts \cite{saad-eigbook, xi-ration}.
In this section, we provide theoretical justifications of the key ideas in \cite{xi-ration} and also suggest some new improvements.

In \cite{xi-ration}, the authors only consider  circular contours. Here, we first generalize their results to arbitrary contours. Let $\gamma$ be a closed piece-wise smooth Jordan curve in the complex plane $\mathbb{C}$ that encloses the origin and such that no eigenvalue of $A$ lies on $\gamma$.
Then the eigenprojector $P$ associated with 
the eigenvalues \emph{outside} $\gamma$ can be expressed as
\begin{equation}
P =\frac{1}{2\pi\mathrm{i}}
\int_\gamma\left(I-zA^{-1}\right)^{-1}\frac{\mathrm{d}z}{z},
\label{eq:eigenprojector}
\end{equation}
where the integral is taken counter-clockwise on $\gamma$.
%
This is because
\begin{equation}
\label{eq:eigenprojector2}
P=\frac{-1}{2\pi\mathrm{i}}\int_\gamma
\left(\frac{1}{z}-A^{-1}\right)^{-1}\mathrm{d}\frac{1}{z}
= \frac{1}{2\pi\mathrm{i}}\int_{\gamma^{-1}}
\left(z'I - A^{-1}\right)^{-1}\mathrm{d}z',
\end{equation}
where $\gamma^{-1} = \{z^{-1}: z\in\gamma\}$ and the last integral is taken counter-clockwise on $\gamma^{-1}$.
The right-hand side of \eqref{eq:eigenprojector2} takes the standard form of an eigenprojector of $A^{-1}$ associated with eigenvalues enclosed by $\gamma^{-1}$,
see for example \cite[Theorem 3.3]{saad-eigbook}.
Assuming that the
$1/\lambda_i$'s are those eigenvalues of $A^{-1}$ enclosed by $\gamma^{-1}$, we then have
\[
\operatorname{Range}(P)=
\bigoplus_i\operatorname{Null}\left(\frac{1}{\lambda_i} I - A^{-1}\right)^{l_i}
=\bigoplus_i\operatorname{Null}\left({\lambda_i} I - A\right)^{l_i},
\]
where $l_i$ is the \emph{index} of $\lambda_i$
\cite[Sec. 1.8.2]{saad-eigbook}
The above equality implies that $P$ in \eqref{eq:eigenprojector} is equal
to the spectral projector of $A$ associated with eigenvalues outside $\gamma$.

The method proposed in \cite{xi-ration} is based on the 
Cauchy integral representation of $PA^{-1}$:
\begin{equation}
PA^{-1}=\frac{1}{2\pi\mathrm{i}}
\int_\gamma\left(A-zI\right)^{-1}\frac{\mathrm{d}z}{z}.
\label{eq:rational}
\end{equation}
Again, the integral is taken counter-clockwise on $\gamma$.
$PA^{-1}$ ignores the eigenvalues of $A$ that are inside $\gamma$,
and attempts to solve the restricted problem for eigenvalues outside $\gamma$.

If a numerical quadrature rule is applied to discretize the right hand side of \eqref{eq:rational},
$PA^{-1}$ can be approximated as
\begin{equation}
PA^{-1} \approx \sum_i \frac{\sigma_i}{z_i} (A-z_i I)^{-1},
\end{equation}
where $\{z_i\}$ are the quadrature nodes along $\gamma$ and $\{\sigma_i\}$ are the 
corresponding weights.

For a given right-hand side $f$, the method proposed in \cite{xi-ration}
first approximates $PA^{-1} f \approx w := \sum_i \frac{\sigma_i}{z_i} (A-z_i I)^{-1}f$
and then tries to minimize the residual of the solution in the range of $I-P$ with an iterative method
\begin{equation}
\min_v \|Av - (f-Aw)\|_2.
\label{eq:inner-problem}
\end{equation}
Afterwards, $v+w$ serves as an approximate solution.
The problem \eqref{eq:inner-problem} is easier to solve 
than the original problem because the spectrum is restricted inside $\gamma$. This framework has some flexibilities regarding the selection of contours and quadrature points and is summarized in 
Algorithm \ref{alg:fci}. Since the linear system is not solved to high accuracy in a single run of {\sf FCI}, it is necessary to use flexible preconditioned GMRES \cite{saad-fgmres}
or iterative refinement to improve the accuracy.

\begin{algorithm}
[!ptbh]\caption{Fast contour integration approximation of $A^{-1}f$}
\label{alg:fci} \tabcolsep=0.7mm\renewcommand{\arraystretch}{1.05}
\begin{algorithmic}
[1]\Procedure{\sf FCI }{$f\in\mathbb{C}^n,A\in\mathbb{C}^{n\times n},
\{z_i\in\gamma\},\{\sigma_i\in\mathbb{C}\}$}

\Comment{$z_i$ and $\sigma_i$ are quadrature points and weights on a contour $\gamma$}
\State{Solve $(A - z_i I)y_i = f$ for each quadrature point $z_i$}
\State{Approximate $PA^{-1}f$ with a quadrature
\[
w = \sum_i \frac{\sigma_i}{z_i} y_i
\]
}
\State{Compute the step size $d$ as follows to compensate quadrature error
\[d=\operatorname{argmin}_{d\in\mathbb{C}} \|f-dAw\|_2\]
}
\State{Solve $v=\operatorname{argmin}_v \|Av -(f-dAw)\|_2$ with 
a few steps of \textsf{GMRES}}
\State{\Return the approximate solution $v+dw$ }
\EndProcedure
\end{algorithmic}
\end{algorithm}

In the following sections, we will discuss how to maximize the efficiency of Algorithm \ref{alg:fci} to solve \eqref{eq:lin-sys} by exploiting the spectral properties of the discretized Helmholtz operators. 

\section{Eigenvalue distribution of the discretized Helmholtz operators}\label{sec:spectrum}
In order to apply the {\sf FCI} preconditioner
(Algorithm \ref{alg:fci}) to solve the linear system \eqref{eq:lin-sys}, the spectrum information of the coefficient matrix is crucial for the selection of the contour as well as the resulting preconditioning effect. In this section, we will systematically study the eigenvalue distribution 
of the discretized Helmholtz operators as well as some of its variants. 
More specifically, we will investigate the spectrum of two types of matrices 1) the coefficient matrix $A$ in \eqref{eq:lin-sys} and 2) a related double-size matrix.

\textbf{Case One:}
For this simplest case, Algorithm \ref{alg:fci} is applied to solve \eqref{eq:lin-sys} directly.
The skew-Hermitian part of the coefficient matrix $A$ comes from 
absorbing boundary conditions or various types of damping.
Here we assume the skew-Hermitian part of $A$ 
is $-\mathrm{i}$ multiplied by a positive semi-definite matrix. That is,
\[
A = A_1 - \mathrm{i} A_2,
\]
where both $A_1$ and $A_2$ are Hermitian, and $A_2$ is positive semi-definite. This assumption has previously appeared in \cite[Equation (12)]{vangijzen07},
and also in \cite[Equation (1.7)]{gander-smallshift}. Under this assumption, it is easy to characterize the spectrum of $A$ as follows.

\begin{lemma}
If the Hermitian matrices $A_1, A_2$ satisfy $A_1+I\succeq 0$ and $A_2\succeq 0$, then the spectrum of $A=A_1 - \mathrm{i} A_2$ is contained in the closed rectangle
\[
\{\lambda\in\mathbb{C} : \operatorname{Re}(\lambda)\in[-1,\rho_1-1],
\quad \operatorname{Im}(\lambda)\in[-\rho_2,0] \},
\]
where $\rho_1=\rho(A_1+I)$ and $\rho_2=\rho(A_2)$.
\label{lem:spectrum-1x1}
\end{lemma}

\begin{proof}
Let $v$ be a unit right eigenvector of $A$. Then the corresponding eigenvalue $\lambda$ satisfies
\[
\lambda = v^HAv
= v^HA_1v - \mathrm{i}v^HA_2v.
\]
This implies that
\[
\operatorname{Re}(\lambda) =  v^HA_1v \in [-1, \rho_1-1], \quad
\operatorname{Im}(\lambda) = -v^HA_2v \in [-\rho_2, 0].
\]
\hfill
\end{proof}

\begin{remark}
One can always normalize a matrix to make the assumption $A_1+I\succeq 0$ hold.  
In Lemma \ref{lem:spectrum-1x1}, $\rho_1$ and $\rho_2$ represent the horizontal 
and vertical stretch of the spectrum, respectively, and $\rho_1/\rho_2$
measures the aspect ratio of this rectangle.
\end{remark}

\textbf{Case Two:}
We now consider a double-size linear system:
\begin{equation}
(\mathrm{i}C -I)
\begin{pmatrix}
\mathrm{i} u \\ u
\end{pmatrix}
=
\begin{pmatrix}
0\\f
\end{pmatrix},\quad
C = \begin{pmatrix}
       & I\\
-(A_1+I) & -A_2
\end{pmatrix}.
\label{eq:mat-2x2}
\end{equation}
One can check that $u$ in \eqref{eq:mat-2x2} is exactly the solution of \eqref{eq:lin-sys}.
We can apply Algorithm \ref{alg:fci} to solve the system \eqref{eq:mat-2x2} instead.
Although the size of the coefficient matrix $\mathrm{i}C-I$ in \eqref{eq:mat-2x2} is twice as large as that of \eqref{eq:lin-sys}, it could be less costly to solve \eqref{eq:mat-2x2} than \eqref{eq:lin-sys}.
This is because the spectrum of $\mathrm{i}C-I$ can be more compact than that of $A$ under some discretization schemes,
which is analyzed in the following theorem. 

\begin{theorem}
Following the same assumption as in Lemma \ref{lem:spectrum-1x1},
the spectrum of the matrix $C$ defined 
in \eqref{eq:mat-2x2} is contained in 
\[
\left\{
\mu\in\mathbb{C}: |\mu|\leq\frac{\rho_2}{2} + \sqrt{\left(\frac{\rho_2}2\right)^2 + \rho_1},
\quad \operatorname{Re}(\mu)\in [-\rho_2, 0]\right\},
\]
where $\rho_1$ and $\rho_2$ are defined in Lemma \ref{lem:spectrum-1x1}.
Furthermore, if $\rho_2=0$, then the set of eigenvalues of $C$ is
\[
\{\pm\mathrm{i}\sqrt{\lambda_i}: \lambda_i \text{ is an eigenvalue of }A_1+I\}.
\]
\label{thm:spectrum-2x2}
\end{theorem}

\begin{proof}
If $\mu$ is a non-zero eigenvalue of $C$, then the Schur complement $S$ of $\mu I - C$ 
\[
S = \mu I +  A_2 + \mu^{-1} (A_1+I)
\]
has to be singular. 

Denote $\mu S$ by $E$. For a non-zero vector $v$ in the null space of $E$, we have 
\[
0=\left|\frac{v^HEv}{v^Hv}\right|
\geq |\mu^2| - |\mu|\rho_2 - \rho_1.
\]
Thus, $|\mu|\leq\frac{\rho_2}{2} + \sqrt{\left(\frac{\rho_2}2\right)^2 + \rho_1}$.

If $\mu$ is real, then for any vector $w$, we have
\[
\left|\frac{w^HEw}{w^Hw}\right|
\geq |\mu^2| - |\mu|\rho_2.
\]
This implies that $\mu\in[-\rho_2, 0]$. Because otherwise $E\succ 0$ is a contradiction.
If $\mu$ has a non-zero imaginary part,
then the Hermitian part of $E/(\operatorname{Im}(\mu)\mathrm{i})$ is
\[
\frac{\operatorname{Im}(\mu^2)}{\operatorname{Im}(\mu)}I + A_2 = 
2\operatorname{Re}(\mu)I + A_2,
\]
which is positive definite for $\operatorname{Re}(\mu)> 0$
and is negative definite for $\operatorname{Re}(\mu)< -\rho_2/2$.
Therefore, $\operatorname{Re}(\mu)\in [-\rho_2, 0]$.

Finally, we consider the special case when $\rho_2=0$. Let $V^{-1}\Lambda V$ be the diagonalization of $A_1+I$.
If $\rho_2=0$, then $A_2=0$ and we have
\[
\begin{pmatrix}
V&\\&V
\end{pmatrix}
C
\begin{pmatrix}
V^{-1}&\\&V^{-1}
\end{pmatrix}
=
\begin{pmatrix}
 &  I\\
-\Lambda &
\end{pmatrix}.
\]
For the characteristic polynomial, we have
\[
\det(\mu I-C)
=
\det\begin{pmatrix}
\mu I &  -I\\
\Lambda & \mu I
\end{pmatrix} = \prod_i (\mu^2+\lambda_i).
\]
This shows $\{\pm\mathrm{i}\sqrt{\lambda_i}\}$ are the eigenvalues of $C$ when $\rho_2 = 0$.
\hfill
\end{proof}

\begin{remark}
Using the notation and result of Theorem \ref{thm:spectrum-2x2},
the spectrum of $\mathrm{i}C-I$ is contained in
\[
\left\{
\mu\in\mathbb{C}:
|\mu+1|\leq\frac{\rho_2}{2} + \sqrt{\left(\frac{\rho_2}2\right)^2 + \rho_1},\quad
-\rho_2\leq\operatorname{Im}(\mu)\leq 0
\right\}.
\]
We can compare $\mathrm{i}C-I$ with the matrix $A$ in terms of the spreading of spectrum. 
The spectrum of $\mathrm{i}C-I$ is contained in a 
$\rho_2 + \sqrt{\rho_2^2 + 4\rho_1}$ by $\rho_2$ rectangle;
but the spectrum of $A$ is contained in a
$\rho_1$ by $\rho_2$ rectangle.
$A$ can have a more elongated spectrum when $\rho_1/\rho_2$ is large.
Therefore, although $\mathrm{i}C-I$ is double in size, it may be more suitable 
for iterative solvers because the spectrum is less spread out. See Figure \ref{fig:Helmholtz-eigs} for a 2D example.
\end{remark}

\begin{figure}[ptbh]
\centering
\includegraphics[width=\textwidth, trim={1in 0 1in 0}, clip]{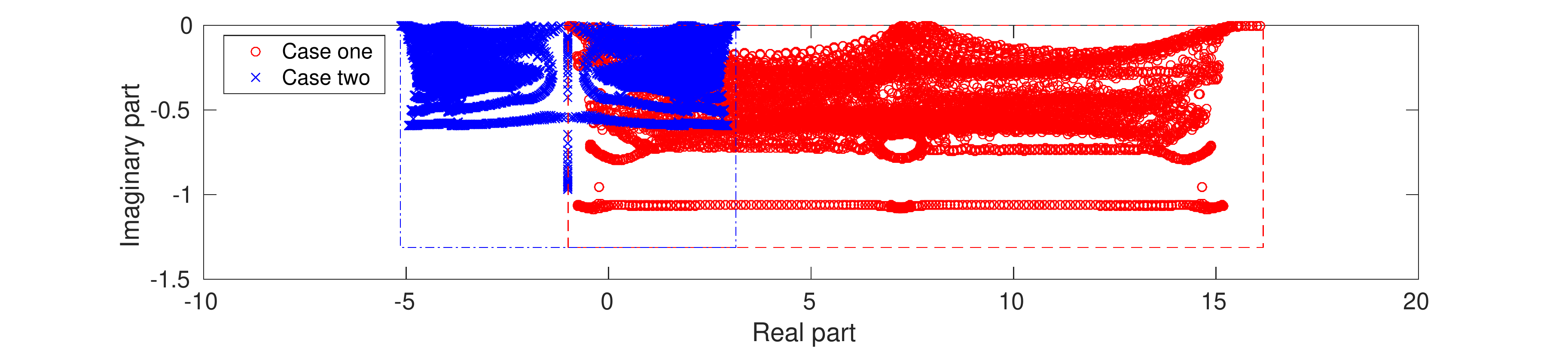}
\caption{Comparison of the spectrum between Case one (\ref{eq:lin-sys}) and Case two (\ref{eq:mat-2x2}).
The test matrix is based on finite difference method on a $100^2$ grid,
with absorbing layers near the boundary.
The eigenvalues are computed by \textsf{eig} in MATLAB.
The rectangles are the results of Lemma \ref{lem:spectrum-1x1} and Theorem \ref{thm:spectrum-2x2},
where $\rho_1\approx 17.2$, $\rho_2\approx 1.3$.}
\label{fig:Helmholtz-eigs}
\end{figure}

\section{Polynomial preconditioners for solving shifted problems}\label{sec:poly}
The application of Algorithm \ref{alg:fci} to solve discretized Helmholtz equations involves several linear system solutions with shifted problems:
\begin{equation}
(A-zI)y = f.
\label{eq:shifted-lin-sys}
\end{equation}
If $|\operatorname{Im}(z)|$ is large enough, then according to the previous section
$\mathrm{i}(A-zI)$ has a sign definite Hermitian part,
and many existing elliptic solvers or preconditioners can be used.
However, they become expensive to compute or
to store as $|\operatorname{Im}(z)|$ reduces.
In this section, we will propose several efficient polynomial preconditioning techniques to solve \eqref{eq:shifted-lin-sys} even when $|\operatorname{Im}(z)|$ is relatively small.

We can write the general form of a \emph{polynomial fixed-point iteration} of \eqref{eq:shifted-lin-sys} as
\begin{equation}
y^{(m+1)} = y^{(m)} + p(A-zI)r^{(m)},
\label{eq:fix-pt-iter}
\end{equation}
where $p$ is a polynomial,
and $r^{(m)}=f-(A-zI)y^{(m)}$ is the residual at the $m$th step.
If all the roots of the polynomial $p$ are known explicitly,
then \eqref{eq:fix-pt-iter} can be rewritten
as a cyclic Richardson iteration.
Motivated from Lemma \ref{lem:spectrum-1x1} in Section \ref{sec:spectrum},
we assume in this section that the spectrum of $A$ is contained in a closed rectangle $\mathcal{B}$.
For fixed $z$, we define
\begin{equation}
R(\lambda) = 1-(\lambda-z)p(\lambda-z).
\label{eq:res-poly}
\end{equation}
$R$ is usually called the \emph{residual polynomial}.
A desirable polynomial $p$ should solve
the following minimax problem in some special set of polynomials denoted by $\mathcal{P}$:
\begin{equation}
\min_{p\in\mathcal{P}}\max_{\lambda\in \mathcal{B}}|R(\lambda)|
=
\min_{p\in\mathcal{P}}\max_{\lambda\in \partial\mathcal{B}}|R(\lambda)|.
\label{eq:minimax}
\end{equation}

\subsection{Stationary Richardson iteration}\label{sub:richardson}
The scheme \eqref{eq:fix-pt-iter} is a stationary Richardson iteration if $p$ is a complex constant.
The minimax problem \eqref{eq:minimax} is essentially solved by the method in \cite{opfer}.
Here, we present a slight generalization to the case of rectangular spectrum.
The following theorem shows how to choose the complex constant
for optimal convergence rate.

\begin{theorem}
Let $\mathcal{B}\subset\mathbb{C}$ be the rectangle with vertices
$\{b_1, b_2, \beta_1, \beta_2\}$ such that
\[
\operatorname{Im}(b_j)=0, \quad \operatorname{Im}(\beta_j) < 0,\quad j\in\{1,2\}.
\]
Assume that $p\in\mathbb{C}$,
the minimax value of $R(\lambda)$ in (\ref{eq:res-poly})--(\ref{eq:minimax})
is taken on the vertices
\[
\min_{p\in\mathbb{C}}\max_{\lambda\in\mathcal{B}}|R(\lambda)|
=\min_{p\in\mathbb{C}}\max_{\lambda\in\{b_1, b_2, \beta_1, \beta_2\}}|R(\lambda)|.
\]
Furthermore, if $\operatorname{Im}(z)\not\in[\operatorname{Im}(\beta_1), 0]$ and $z$
is enclosed by the circumcircle of $\mathcal{B}$,
then the minimax value of $R(\lambda)$ equals $\frac{|\alpha_1 - \alpha_2|}{|\alpha_1|+|\alpha_2|}$
at $p^\ast=\frac{|\alpha_1|/\alpha_1 + |\alpha_2|/\alpha_2}{|\alpha_1|+|\alpha_2|}$,
where
\[
\alpha:= (\alpha_1, \alpha_2) =
\left\{
\begin{split}
(b_1-z,b_2-z) & \quad \text{ if } \operatorname{Im}(z)>0,\\
(\beta_1-z,\beta_2-z) & \quad \text{ otherwise}.\\
\end{split}
\right.
\]
\label{thm:conv-richardson}
\end{theorem}

\begin{proof}
Let $o = 1/p + z$ be the root of the residual polynomial $R(\lambda) = 1-(\lambda-z)p$.
The absolute value of $R(\lambda)$ is related to the distance from $o$
\[
|R(\lambda)| = |1-(\lambda-z)p| = \left|1-\frac{\lambda-z}{o-z}\right|
= \frac{|\lambda - o|}{|z-o|}.
\]
For fixed $z$ and $o$, $|R(\lambda)|$ is convex,
so the maximum value on each line segment is taken on a vertex.
$\partial\mathcal{B}$ has four sides,
and the maximum value is taken on one of the four vertices.

For the remaining part of the theorem,
it suffices to prove for the case $\operatorname{Im}(z)>0$
because the other case follows from symmetry.
The minimax problem restricted to a line segment is solved in \cite[Example 5.1]{opfer},
which suggests
\[
\min_{p\in\mathbb{C}}\max_{\lambda\in [b_1,b_2]} |1-(\lambda-z) p |
= \frac{|\alpha_1 - \alpha_2|}{|\alpha_1|+|\alpha_2|} < 1,
\]
and the optimal value of $p$ is $p^\ast$.
For circles containing $b_1$, $b_2$,
the one centered at $o^\ast = 1/p^\ast + z$ does not contain $z$
because $|b_1 - o^\ast|=|b_2 - o^\ast|<|z-o^\ast|$,
but the circumcircle of $\mathcal{B}$ contains $z$.
So, $o^\ast$ is closer to $\beta_1, \beta_2$ than to $b_1, b_2$. That is,
\[
|\beta_1 - o^\ast| = |\beta_2- o^\ast|< |b_1 - o^\ast| = |b_2 - o^\ast|.
\]
Therefore,
\[
\frac{|b_1-o^\ast|}{|z-o^\ast|}
=
\min_{p\in\mathbb{C}}\max_{\lambda\in\{b_1, b_2, \beta_1, \beta_2\}}|R(\lambda)|
=\min_{p\in\mathbb{C}}\max_{\lambda\in\mathcal{B}}|R(\lambda)|.
\]
$p^\ast$ solves the minimax problem \eqref{eq:minimax} for $\mathcal{B}$.
\hfill\end{proof}

\begin{remark}
If $b_1b_2<0$, which means the matrix $A$ in \eqref{eq:shifted-lin-sys} is indefinite,
for an imaginary shift $z$ (or $|\operatorname{Re}(z)|$ is much smaller than $|b_1|$ or $|b_2|$),
the convergence rate is close to
\[
\frac{|\alpha_1 - \alpha_2|}{|\alpha_1|+|\alpha_2|}
= \frac{|b_1|+|b_2|}{\sqrt{b_1^2 + \operatorname{Im}^2(z)} + \sqrt{b_2^2 + \operatorname{Im}^2(z)}}.
\]
Using the Taylor expansion of $\sqrt{1+x^2}$,
one can check that an $O(|\operatorname{Im}(z)|^{-2})$ number of iterations
is needed to reach certain relative accuracy.
This result can be improved by considering high-order polynomials.
\end{remark}

\subsection{High-order polynomials}\label{sub:high-order}
High-order polynomials $p(\lambda-z)$ have the capability to improve the convergence rate.
For well-conditioned problems,
existing work such as Chebyshev iterations
\cite{wrigley, manteuffel77,manteuffel78,gutknecht}
and Leja points \cite{reichel} can select
the roots of the polynomial near the spectrum for asymptotic optimal convergence.
Since the shifted system \eqref{eq:shifted-lin-sys} may not be sufficiently well conditioned,
we will design such a polynomial from the approximation of the exponential function.

For given $z$, the residual polynomial $R(\lambda)$ defined in \eqref{eq:res-poly}
can be reformulated as
\begin{equation}
R(\lambda) = \frac{\tilde{p}(\lambda)}{\tilde{p}(z)},
\end{equation}
for some polynomial $\tilde{p}$
because this form also takes the value of $1$ at $\lambda=z$.
The ideal polynomial should yield small value of $|R(\lambda)|$ at every eigenvalue $\lambda$.

The choice of $\tilde{p}$ is motivated from the exponential function.
If $\operatorname{Im}(\lambda - z)$ has a fixed sign for $\lambda\in\mathcal{B}$,
then with a suitable choice of $\delta\in\mathbb{R}$,
the following quantity can be arbitrarily small
\[
\frac{|e^{-\mathrm{i}\delta\lambda|}}{|e^{-\mathrm{i}\delta z}|} =
e^{\delta\operatorname{Im}(\lambda - z)}.
\]
The simplest choice of $\tilde{p}$ is based on the Taylor expansion of the exponential function
\begin{equation}
\tilde{p}^{(q)}(\lambda) = \sum_{j=0}^{q}\frac{(-\mathrm{i}\delta(\lambda-z_0))^j}{j!},
\label{eq:exp-taylor}
\end{equation}
where $z_0$ is the center of the Taylor expansion.
Then the polynomial $p$ of degree $q-1$ in \eqref{eq:fix-pt-iter} has the form 
\[
p(\lambda-z) = \left(1-\frac{\tilde{p}^{(q)}(\lambda)}{\tilde{p}^{(q)}(z)}\right)/(\lambda-z).
\]
Since it is complicated to expand the above expression in the power basis, we propose an alternative iteration scheme as follows:

\begin{equation}
\begin{split}
k_1 & = -\mathrm{i}\delta\left((A-z_0I)y^{(m)} - f\right),\\
k_j & = \frac{-\mathrm{i}\delta}{j}
\left((A-z_0I)k_{j-1}-\frac{(-\mathrm{i}\delta(z-z_0))^{j-1}}{(j-1)!}f\right),\quad
j\in\{2,3,\cdots,q\},\\
y^{(m+1)} & = \frac{y^{(m)} + \sum_{j=1}^{q} k_j }{\tilde{p}^{(q)}(z)}.
\end{split}
\label{eq:fix-pt-exp}
\end{equation}

In the following theorem, we show that the above scheme is equivalent to a fixed point iteration.
\begin{theorem}
The solution of (\ref{eq:shifted-lin-sys}) is a fixed point of (\ref{eq:fix-pt-exp}).
The error at the $m$th step satisfies
\[
y^{(m)} - y
=
\left(
\frac{\tilde{p}^{(q)}(A)}
{\tilde{p}^{(q)}(z)}
\right)^m
\left(
y^{(0)} - y
\right).
\]
\label{thm:fix-pt-exp}
\end{theorem}
\begin{proof}
If $y^{(m)} = y$ is the solution of (\ref{eq:shifted-lin-sys}),
then $k_1 = -\mathrm{i}\delta(z-z_0)y$.
If $k_{j-1} = \frac{(-\mathrm{i}\delta(z-z_0))^{j-1}}{(j-1)!}y$,
then
\begin{align*}
k_{j} &=
\frac{-\mathrm{i}\delta}{j}\left((A-z_0I)k_{j-1}-\frac{(-\mathrm{i}\delta(z-z_0))^{j-1}}{(j-1)!}g
\right)\\
&=\frac{(-\mathrm{i}\delta(z-z_0))^{j-1}}{(j-1)!}
\frac{-\mathrm{i}\delta}{j}\left((A-z_0I)y - g\right)\\
&=
\frac{(-\mathrm{i}\delta(z-z_0))^j}{j!}y.
\end{align*}
For the next iterate,
\[
y^{(m+1)}
=
\frac{\sum_{j=0}^{q}(-\mathrm{i}\delta(z-z_0))^j/j!}
{\tilde{p}^{(q)}(z)} y = y.
\]
Therefore, $y$ is a fixed point of (\ref{eq:shifted-lin-sys}).

If $y^{(m)} \neq y$,
then from a similar induction argument we have
\[
k_{j} - \frac{(-\mathrm{i}\delta(z-z_0))^j}{j!}y
=\frac{(-\mathrm{i}\delta(A-z_0I))^j}{j!}
\left(y^{(m)}-y\right), \quad 1\leq j\leq q.
\]
Therefore,
\[
y^{(m+1)} - y
=
\frac{\tilde{p}^{(q)}(A)}
{\tilde{p}^{(q)}(z)}
\left(
y^{(m)} - y
\right).
\]
\hfill\end{proof}

The optimal parameters $z_0$ and $\delta$ in the scheme \eqref{eq:fix-pt-exp} can be computed by solving an optimization problem for each fixed polynomial degree $q$. Here we propose a heuristic to simplify this procedure. We choose $z_0$ to guarantee robustness and $\delta$ for fast convergence.
By robustness we mean for sufficiently small $|\delta|$,
the spectral radius
$
\rho(\tilde{p}^{(q)}(A)/\tilde{p}^{(q)}(z))
$
is less than $1$. This is done by considering the following equation:
\[
\left|\frac{\tilde{p}^{(q)}(\lambda)}{\tilde{p}^{(q)}(z)}\right|^2=
\left|
\frac{1-\mathrm{i}\delta(\lambda-z_0)}
{1-\mathrm{i}\delta(z-z_0)}
\right|^2
+o(|\delta|)
=
\frac{1+2\delta\operatorname{Im}(\lambda-z_0)}
{1+2\delta\operatorname{Im}(z-z_0)}
+o(|\delta|).
\]
If we fix $\operatorname{Im}(z_0) = \operatorname{Im}(z)$ 
and assume $\operatorname{Im}(\lambda-z_0)$ has a fixed sign,
then we can always find some $\delta$ with a small absolute value
such that $\delta\operatorname{Im}(\lambda-z_0)<0$,
which controls the spectral radius. Since $z_0 = (b_1+b_2)/2+\mathrm{i}\operatorname{Im}(z)$ is the optimal choice for $q=1$, we will always follow this choice for high order polynomials. 
After that, $\delta$ is determined by numerically minimizing the convergence rate
\begin{equation}
\nu = \min_{\delta\in\mathbb{R}}\max_{\lambda\in\partial\mathcal{B}}
\left|\frac{\tilde{p}(\lambda)}{\tilde{p}(z)}\right|.
\label{eq:fix-pt-conv}
\end{equation}
This is a 2D optimization problem which is easy to solve.
Table \ref{tab:exp-rate-order} compares the convergence rate $\nu$ for different order $q$.
We prefer choosing $q$ with a minimum $\nu^{1/q}$ value since this quantity gives the fastest converging method for a given number of matrix-vector products.

\begin{table}[ptbh]
\caption{Convergence rate of (\ref{eq:fix-pt-exp}) for different order $q$.
The spectrum of $A$ in (\ref{eq:shifted-lin-sys}) is
within a rectangle $\mathcal{B}$
with the real part in $[-1,2.8]$
and the imaginary part in $[-0.65,0]$,
and the shift is $z = \mathrm{i}$.
$\delta^\ast$ is the optimal choice of $\delta$.
$\nu$ is the estimated convergence rate,
and $\nu^{1/q}$ quantifies the convergence rate per matrix-vector product.
}\label{tab:exp-rate-order}
\begin{center}
\tabcolsep=1.6mm\renewcommand{\arraystretch}{1.1}
\begin{tabular}
[c]{|c|c|c|c|}
\hline
$q$ & $\delta^\ast$ & $\nu$ &$\nu^{1/q}$ \\
\hline
1 & 0.250 & 0.866 & 0.866 \\
\hline
2 & 0.688 & 0.537 & 0.733 \\
\hline
3 & 0.750 & 0.530 & 0.810 \\
\hline
4 & 0.750 & 0.547 & 0.860 \\
\hline
5 & 0.938 & 0.416 & 0.839 \\
\hline
\end{tabular}
\end{center}
\end{table}

We are also concerned with how the cost increases
as $|\operatorname{Im}(z)|$ reduces in \eqref{eq:shifted-lin-sys}.
Table \ref{tab:exp-nmatvec} gives the estimated number
of matrix-vector products for reducing the residual by $10^{2}$.
The cost of the stationary Richardson iteration ($q=1$) quadruples
as the imaginary shift reduces by 1/2.
For high-order methods such as $q=3$, the results are much better.
The cost roughly doubles when the imaginary shift reduces by 1/2.
As the imaginary shift decreases,
one might want to increase $q$ slightly to approach a desirable performance.

\begin{table}[ptbh]
\caption{Estimated number of matrix-vector products for reducing the residual by $10^{2}$.
The spectrum of $A$ is the same as Table \ref{tab:exp-rate-order}.}
\label{tab:exp-nmatvec}
\begin{center}
\tabcolsep=1.6mm\renewcommand{\arraystretch}{1.1}
\begin{tabular}
[c]{|c|c|c|c|c|}
\hline
& $z=\mathrm{i}$ & $z=\mathrm{i}/2$ & $z=\mathrm{i}/4$ & $z=\mathrm{i}/8$\\
\hline
$q=1$ & 33 & 110 & 420 & 1656\\
\hline
$q=2$ & 16 &  38 & 110 &  304\\
\hline
$q=3$ & 24 &  30 &  60 &  135\\
\hline
$q=4$ & 32 &  44 &  72 &  128\\
\hline
$q=5$ & 30 &  60 & 120 &  180\\
\hline
\end{tabular}
\end{center}
\end{table}

\section{Solving the inner problem}\label{sec:inner}
After the discussion of the outer problem,
the next topic is the convergence of GMRES iterations 
for the inner problem in Algorithm \ref{alg:fci}.
We analyze the problem by studying 
the distribution of eigenvalues with small magnitude for the Helmholtz equation.

\subsection{Characterization of small eigenvalues}
It is important to study near-zero eigenvalues because 
they govern the conditioning of the problem.
A stability estimate is proved in \cite{melenk, cummings, gander-smallshift}
for the constant-coefficient Helmholtz equation
in a star-shaped domain $\Omega$ with the impedance boundary condition.
The part of the statement we want to highlight is that
\[
\omega \|u\|_{L^2(\Omega)} \leq \alpha \|f\|_{L^2(\Omega)},
\]
where $u$ is the solution of the right-hand side $f$,
$\alpha$ is a constant independent from the angular frequency $\omega$.
This result implies that \emph{all the eigenvalues of $A$ are at least $O(\omega)$ distance away from the origin.}

For the constant-coefficient case in a hypercube,
we can further describe the distribution of eigenvalues with the impedance boundary condition.
In fact, we can show that \emph{the imaginary part is at least $O(\omega)$
for eigenvalues with magnitude smaller than $\omega^2$.}
Let the wavespeed be normalized as one in the $d$-dimensional unit hypercube $[0,1]^d$.
Consider the eigenvalue problem in multiple dimensions ($d\geq 2$),
and an eigenpair $(\lambda,v)$ satisfies
\begin{equation}
\left\{
\begin{split}
-\Delta v - \omega^2 v &= \lambda v\quad \text{in }[0,1]^d,\\
\partial_n v - \mathrm{i}\omega v &= 0\quad \text{on } \partial[0,1]^d,
\end{split}
\right.
\label{eq:eig-ibc}
\end{equation}
where $\partial_n$ means taking the directional derivative along the outward unit normal.

\begin{theorem}
For any fixed $\rho\in (0,1)$,
there exists a positive constant $s$ 
such that every eigenvalue $\lambda$ of (\ref{eq:eig-ibc}) with $|\lambda|\leq\rho\omega^2$
satisfies $|\operatorname{Im}(\lambda)|\geq s \omega$ for sufficiently large $\omega$.
\label{thm:dist-eig-ibc}
\end{theorem}

\begin{figure}[ptbh]
\centering
\includegraphics[width=0.8\textwidth, trim={0 1in 0 0}, clip]{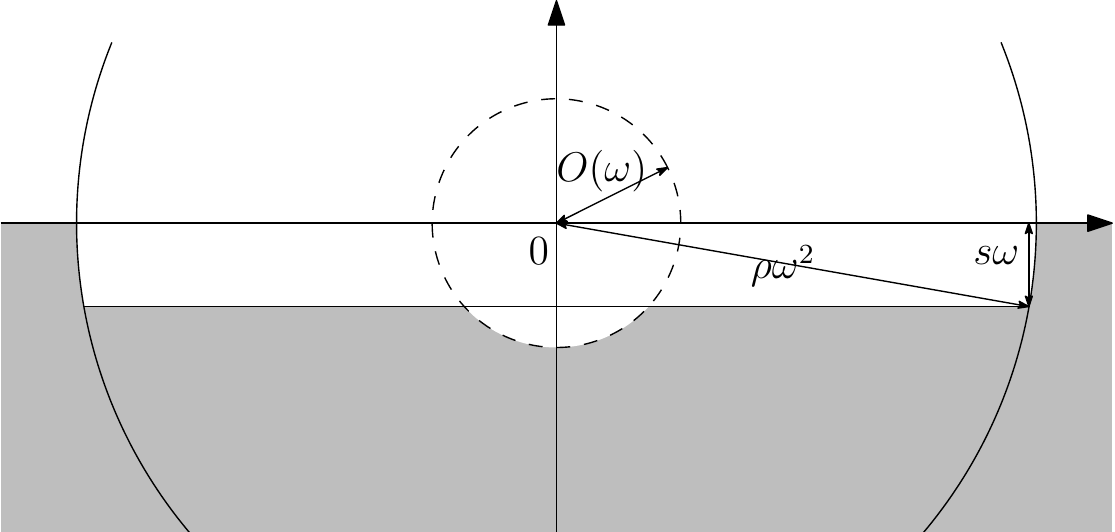}
\caption{Distribution of eigenvalues in the complex plane 
for the Helmholtz equation with impedance boundary condition.
All the eigenvalues are in the shaded area of the lower half plane.
The dashed circle around the origin with radius $O(\omega)$
is the result of \cite{melenk, cummings, gander-smallshift}.
In the circle with radius $\rho\omega^2$, 
the minimum $s\omega$ distance from the real axis 
is the result of Theorem \ref{thm:dist-eig-ibc}.}
\label{fig:eig-est}
\end{figure}

The theorem gives a more detailed characterization of small eigenvalues.
Figure \ref{fig:eig-est} illustrates how this compares with existing stability results.
The proof is based on a separation of variables in following form:
\begin{equation}
\label{eq:sep-ode}
v(x)=\prod_{j=1}^d \varphi_j(x_j), \quad \varphi_{j}^{\prime\prime}(x_{j})
+\xi_{j}^{2} \varphi_{j}(x_{j}) =0,
\end{equation}
where $\xi_j\in\mathbb{C}$ and $\operatorname{Re}(\xi_j)\geq 0$.
An eigenvalue can be written as $\lambda = \xi\cdot\xi - \omega^2$, where
$\xi = \begin{pmatrix}\xi_1 & \xi_2 & \cdots &\xi_d\end{pmatrix}$.
The general solution of \eqref{eq:sep-ode} is
$
\varphi_{j}(x_{j}) = a_{j}^{+}e^{\mathrm{i}\xi_{j} x_{j}} + a_{j}^{-}e^{-\mathrm{i}\xi_{j}x_{j}}.
$
The boundary condition in \eqref{eq:eig-ibc} suggests
\begin{equation}
\label{eq:sep-bc}
-\varphi^{\prime}_{j}(0) = \mathrm{i}\omega\varphi_{j}(0),\quad
 \varphi^{\prime}_{j}(1) = \mathrm{i}\omega\varphi_{j}(1).
\end{equation}
Substituting the general solution into \eqref{eq:sep-bc}, we have
\begin{align}
a_{j}^{+}(\omega+\xi_{j})\phantom{e^{\mathrm{i}\xi_j}}  &  = -a_{j}^{-}(\omega-\xi
_{j}),\\
a_{j}^{+}(\omega-\xi_{j})e^{\mathrm{i}\xi_{j}}  &  = -a_{j}^{-}(\omega+\xi
_{j})e^{-\mathrm{i}\xi_{j}}.
\end{align}
Eliminating $a_{j}^{\pm}$, we have that every $\xi_{j}$ solves the same equation of $z$
\begin{equation}
\label{eq:sep-frac-exp}\frac{\omega-z}{\omega+z}=\pm e^{-\mathrm{i}z}.
\end{equation}
There is no root in the first quadrant $\{z:\operatorname{Re}(z)\geq0,
\operatorname{Im}(z) > 0\}$ because
\[
\frac{|\omega-z|}{|\omega+z|}\leq1,\quad|e^{-\mathrm{i}z}|>1.
\]
So we are only interested in the fourth quadrant
$\{z: \operatorname{Re}(z)\geq0, \operatorname{Im}(z)\leq0\}$.
The following lemma will be used in the proof of Theorem \ref{thm:dist-eig-ibc}.

\begin{lemma}
Given a sequence of angular frequencies $\{\omega_m\}$ that goes to infinity,
and for a sequence of complex numbers in the fourth quadrant
$
\{z_m=a_m-\mathrm{i}b_m:a_m,b_m\geq0\},
$
if there exists $c_{1}, c_{2}>0$ such that $b_{m}\geq c_{1} \omega_{m},
b_{m}\geq c_{2} a_{m}$,
then $z_{m}$ does not solve (\ref{eq:sep-frac-exp}) for sufficiently large $m$.
\label{lem:not-sol}
\end{lemma}

\begin{proof}
The lemma can be easily proved by taking the absolute value on both sides of (\ref{eq:sep-frac-exp}).
For the left-hand side of \eqref{eq:sep-frac-exp}, we have
\[
\left|  \frac{\omega_m-a_m + \mathrm{i}b_m}{\omega_m+a_m - \mathrm{i}b_m}\right|
\geq\frac{b_m}{|\omega_m+a_m - \mathrm{i}b_m|}\geq
\frac{1}{\sqrt{(1/c_{1}+1/c_{2})^{2}+1}}.
\]
The right-hand side of \eqref{eq:sep-frac-exp} satisfies
\[
\lim_{m\to\infty}|\exp(-b_m-\mathrm{i}a_m)|= \lim_{m\to\infty}\exp(-b_m)=0.
\]
So the equality does not hold for sufficiently large $m$.
\hfill
\end{proof}

{\it Proof of Theorem \ref{thm:dist-eig-ibc}.}
If the statement is false, then there exists a $\rho>0$, 
a sequence of angular frequencies $\{\omega_{m}\}$ that goes to infinity,
and a sequence of complex phase vectors
$\{\xi^{(m)}\in\mathbb{C}^d\}$ satisfying \eqref{eq:sep-ode} and \eqref{eq:sep-bc},
but for $\lambda_m = \xi^{(m)}\cdot\xi^{(m)}-\omega_{m}^{2}$,
we have that
\begin{equation}
|\lambda_m| \leq \rho\omega_m^2,\quad
\lim_{m\to\infty} \frac{\operatorname{Im}(\lambda_{m})}{\omega_m} = 0.
\label{eq:lam-norm-imag}
\end{equation}

We can assume that there exists a sequence of indices $\{j_m\in\{1,2,\cdots,d\}\}$
such that 
\begin{equation}
|\xi_{j_m}^{(m)}|\geq\omega_m\sqrt{\frac{1-\rho}d},
\label{eq:xi-norm}
\end{equation}
because otherwise $|\xi^{(m)}|^2 < \omega_m^2(1-\rho)$, and we have 
\[
|\lambda_m|\geq \omega_{m}^{2}-|\xi^{(m)}|^2 >\rho \omega_m^{2},
\]
which contradicts with \eqref{eq:lam-norm-imag}.

Let $\xi_{j_m}^{(m)}=a_m-\mathrm{i}b_m$.
Because of \eqref{eq:sep-frac-exp}--\eqref{eq:xi-norm},
$a_m,b_m\geq0$ satisfy
\begin{align}
\label{eq:sep-seq-norm}
a_m^2+b_m^2 &\geq\frac{1-\rho}{d}\omega_m^2,\\
\label{eq:sep-seq-lim}
\lim_{m\to\infty} \frac{a_mb_m}{\omega_m}&=0,\\
\label{eq:sep-seq-root}
\frac{\omega_m-a_m + \mathrm{i}b_m}{\omega_m+a_m - \mathrm{i}b_m}&=\pm\exp(-b_m-\mathrm{i}a_m).
\end{align}
From \eqref{eq:sep-seq-norm} and \eqref{eq:sep-seq-lim},
we get 
\[
\frac{\omega_ma_mb_m}{a_m^2+b_m^2} =
\frac{a_mb_m}{\omega_m} \frac{\omega_m^2}{a_m^2+b_m^2}\rightarrow0.
\]
Then we can find subsequences (still denoted by $\{a_m\},\{b_m\}$)
such that either $\omega_ma_m/b_m$ or $\omega_mb_m/a_m$ goes to zero.

If $\omega_ma_m/b_m \to0$, then $a_m/b_m\to 0$ for large $\omega_m$. 
From \eqref{eq:sep-seq-norm}, we also have 
\[
\frac{\omega_m}{b_m}\leq \sqrt{\frac{d}{1-\rho}\left(1+\frac{a_m^2}{b_m^2}\right)}.
\]
$\{a_m-\mathrm{i}b_m\}$ satisfies the assumptions in Lemma
\ref{lem:not-sol}. Hence they are not roots of \eqref{eq:sep-frac-exp}
for large $m$, contradiction.

If $\omega_{m}b_{m}/a_{m} \to0$, then from \eqref{eq:sep-seq-lim}
\[
b_{m}^2 = \frac{\omega_{m}b_{m}}{a_m} \frac{a_mb_m}{\omega_m}
\to0.
\]
From \eqref{eq:sep-seq-norm}, we have
\[
\frac{a_m^2}{\omega_m^2}\geq \frac{1-\rho}{d} -\frac{b_m^2}{\omega_m^2}.
\]
So $a_{m}/\omega_{m}$ is bounded above zero.
From \eqref{eq:sep-seq-root},
\[
1 = \exp(-\lim_{m\to\infty} b_m) = \lim_{m\to\infty}\frac{|1-a_m/\omega_m|}{1+a_m/\omega_m}.
\]
So $a_m/\omega_m\to\infty$. For the complex phase vectors $\{\xi^{(m)}\}$,
\[
\lambda_m = (a_m^{2}-b_m^{2}-\omega_m^{2}) - \mathrm{i}2a_mb_m
+\sum_{l\neq j_{m}}\xi_{l}^{(m)}\xi_{l}^{(m)}.
\]
For large $m$, we have $a_{m}^{2}-b_{m}^{2}-\omega_{m}^{2} > 2\rho\omega_{m}^{2}$.
We can find another sequence 
$\{\xi_{l_{m}}^{(m)}: l_{m}\neq j_{m}\}$
such that 
\[
\operatorname{Re}(\xi_{l_m}^{(m)}\xi_{l_m}^{(m)}) < - \frac{\rho}{d-1} \omega_{m}^2.
\]
Because otherwise
\[
\operatorname{Re}(\lambda_{m})>
2\rho\omega_{m}^{2} + \sum_{l\neq j_{m}}\operatorname{Re}(\xi_{l}^{(m)}\xi_{l}^{(m)})
\geq 2\rho\omega_{m}^2 - (d-1)\frac{\rho}{d-1} \omega_{m}^2 =
\rho \omega_{m}^2,
\]
which contradicts with \eqref{eq:lam-norm-imag}.
Let $\xi_{l_{m}}^{(m)}=\tilde{a}_{m}-\mathrm{i}\tilde{b}_{m}$ 
with non-negative $\tilde{a}_{m}$ and $\tilde{b}_{m}$.
We have
\[
\tilde{a}_{m}^2 - \tilde{b}_{m}^2 < - \frac{\rho}{d-1} \omega_{m}^2.
\]
So $\tilde{b}_{m}>\tilde{a}_{m}$ and $\tilde{b}_{m} > \sqrt{\rho/(d-1)}\omega_{m}$.
Because of Lemma \ref{lem:not-sol}, this sequence does not solve
\eqref{eq:sep-frac-exp} for large $m$, which is a contradiction.
\hfill
\endproof

\subsection{GMRES convergence}
Theorem \ref{thm:dist-eig-ibc} depicts the fine structure of spectrum of discretized Helmholtz operators near the origin. 
For simplicity and clarity,
we assume for now that the contour integration \eqref{eq:rational} 
gives the exact solution of the outer problem,
and focus on the convergence of GMRES for solving the inner problem.

For the linear system $Au=f$, recall that $P$ is the eigenprojector defined in \eqref{eq:eigenprojector}.
After solving the outer problem, we obtain the solution $PA^{-1}f$ and the residual
\[
f-APA^{-1}f = (I-P)f.
\]
Since the solution of the inner problem is in the range of $I-P$, we have the following equation
\begin{equation}
A(I-P)u = (I-P)f.
\label{eq:inner}
\end{equation}

\begin{theorem}
Assume that $A(I-P)$ is normal,
and that there exists positive constants $\rho, s$ such that
$|\lambda|\leq \rho\omega^2$ and $\operatorname{Im}(\lambda)\leq - s\omega$ for all non-zero eigenvalues $\lambda$ of $A(I-P)$,
then for solving \eqref{eq:inner}, the $k$th GMRES residual $r_k$ satisfies
\[
\|r_k\|_2\leq \mu^k\|r_0\|_2,\quad \mu^2 \leq 1-\frac{s^2}{\rho^2\omega^2}.
\]
\end{theorem}

\begin{proof}
{
Since $A(I-P)$ is normal, $A(I-P)$ can be unitarily diagonalizabe. That is, there exists a matrix $V$ with orthonormal columns 
such that
\[
A(I-P) = V\Lambda V^H,
\]
where $\Lambda$ is a diagonal matrix containing non-zero eigenvalues.
Then,
\[
I-P = V V^H,
\]
and \eqref{eq:inner} is equivalent with
\[
\Lambda (V^H u) = V^H f.
\]
It is easy to see that the smallest eigenvalue of the Hermitian part of $\mathrm{i}\Lambda$ is greater than or equal to $s\omega$ and $\|\mathrm{i} \Lambda \|_2\leq \rho \omega^2$. Then the final result follows directly from the Elman estimate 
\cite{elman} for the matrix $\Lambda$.
}
\hfill
\end{proof}

\section{Techniques for improved performance}\label{sec:tech}
In this section, we will discuss several special techniques to improve the speed and robustness of Algorithm \ref{alg:fci}.

\subsection{Quadrature on an ellipse}
The contour $\gamma$ is usually selected as a circle in existing methods.
By shrinking, say, the real axis,
the circle is transformed into an ellipse in the form of
\begin{equation}
\left\{
tr\cos\theta + \mathrm{i}r\sin\theta:\ t,r>0,\ \theta\in [0,2\pi]\right\}.
\label{eq:ellipse}
\end{equation}
For fixed $r$, 
the advantage of using a small ratio $t$ is that 
there are fewer eigenvalues enclosed by the contour.
Hence the contour integration \eqref{eq:rational} is closer to the true inverse.
The disadvantage is that the shape of $\gamma$ becomes irregular,
so the number of quadrature points may need to increase.

One way to derive quadrature rules on an ellipse is 
by mapping it to the unit circle parametrized by the angle $\theta$.
Note that
\[
\frac{1}{2\pi\mathrm{i}}\mathrm{d}z = 
\frac{1}{2\pi\mathrm{i}} \mathrm{d}( tr \cos\theta + \mathrm{i}r\sin\theta)
=(r\cos\theta+\mathrm{i} tr\sin\theta) \frac{\mathrm{d}\theta}{2\pi}.
\]
For an equispaced set of angles $\{\theta_j\}$,
the quadrature weights can be chosen as
\[
\sigma_j = (r\cos\theta_j+\mathrm{i} tr\sin\theta_j)/J,
\]
where $J$ is the number of quadrature points.

\subsection{Shifting the center}
In order to avoid an ill-conditioned shifted system,
we shift the center of the ellipse according to the spectrum
and choose a small number of points such as $6$ so that
the quadrature points are not close to any eigenvalue.
This makes the contour integration method robust.

{
Each quadrature point $z_j$ satisfies
\begin{equation}
z_j = tr(\cos(2j-1)\phi - \cos\phi)
+ \mathrm{i}(r\sin(2j-1)\phi - \rho_2/2),\quad j\in[1,\dots,J],
\label{eq:quad-points}
\end{equation}
where $\phi = \pi/J$ and $\rho_2$ is the spectral radius of the skew-Hermitian part of $A$.
The center of the contour is $-tr\cos\phi - \mathrm{i}\rho_2/2$.
For the real part,
\[
\operatorname{Re}(z_j) = tr(\cos(2j-1)\phi - \cos\phi) \leq 0,
\]
so that the shifted matrices cannot be more indefinite than the original one.
We choose 
\begin{equation}
r=(\rho_2/2 + \epsilon)/\sin\phi,
\label{eq:quad-radius}
\end{equation}
so that the imaginary part is
\[
\operatorname{Im}(z_j)\not\in (-\epsilon-\rho_2,\epsilon).
\]
According to Lemma \ref{lem:spectrum-1x1},
this choice ensures that each $A-z_jI$ is invertible.
The quadrature points are visualized in Figure \ref{fig:contour}.
}

\begin{figure}[ptbh]
\centering
\includegraphics[width=\textwidth]{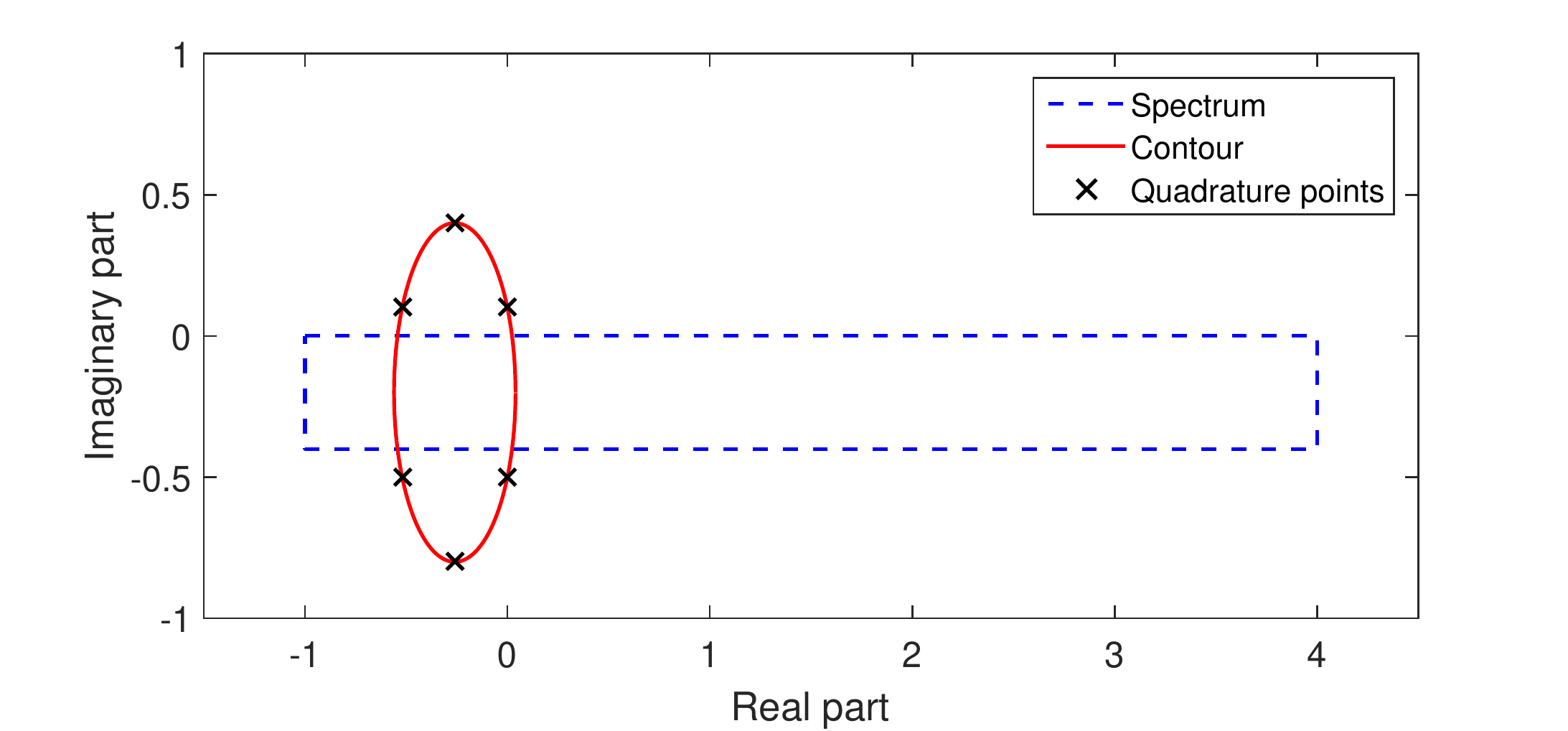}
\caption{Illustration of the quadrature points \eqref{eq:quad-points}.
Here, $t=0.5$, $r = 0.6$, and $J = 6$.}
\label{fig:contour}
\end{figure}

\subsection{Preconditioning the inner problem}
{The inner problem discussed in Section \ref{sec:inner} 
is an idealized case where the outer problem is exactly solved.
Error in the solution and quadrature may affect the convergence of the inner problem.
One can apply some standard preconditioners to improve the convergence.
We find that an (approximate) discrete Laplacian
usually gives a satisfactory performance by deflating
the large eigenvalues of the Helmholtz problem.
For the sparse case, this can be achieved by computing
an ILU factorization without fill.
}

\section{Numerical examples}
To illustrate the performance of the proposed method,
we present the test results for solving
a challenging 3D high-frequency Helmholtz equation.
Since the method has little restrictions on the type of discretization:
we attempt to solve both dense linear systems from a Fourier spectral method
as well as sparse ones based on finite difference methods.
The solution algorithms are implemented by MATLAB.
The test machine is a Linux workstation having 3.5GHz CPU and 64GB RAM.
{In this section, we use the following notation:
\begin{itemize}
\item \textsf{its}: number of iterations;
\item \textsf{mvs}: number of matrix-vector products;
\item \textsf{i-t}: iteration time in seconds.
\end{itemize}
}
\subsection{Test problem}
The test matrix can be written formally as
\begin{equation}
A= S-M + \mathrm{i}D,
\label{eq:stiff-mass}
\end{equation}
for some Hermitian positive semi-definite matrices $S$, $M$, and $D$.
$S$ is the negative discrete Laplacian,
$M$ is the mass matrix that generates the indefiniteness of $A$,
and $D$ gives the non-Hermitian part.
For solving free-space problems,
$D$ is used to reduce artificial reflections near the boundaries of the computational domain.
One example is a diagonal matrix which has positive diagonal entries for points near the boundary 
and zero elsewhere, see for example \cite{shin-sponge}.

If $S, M, D \succeq 0$, and $\rho(M)\leq 1$, then $S-M+I\succeq 0$
and $A$ satisfies the assumptions of Lemma \ref{lem:spectrum-1x1}.
It is helpful to know $\rho_1 = \rho(S-M+I)$ and $\rho_2 = \rho(D)$,
because they characterize the spectrum of $A$.
Since $\rho_2$ only depends on $D$,
we study the combined effect of $S$ and $M$ on $\rho_1$ for our test problem.

Fourier spectral method can be applied because the proposed method 
only requires matrix-vector products of the discrete Laplacian.
$S$ can be diagonalized by a 3D fast Fourier transform (FFT):
\[
S=F^{H}\Lambda F,
\]
where $F$ is the transformation matrix of a forward FFT,
and $\Lambda$ is the diagonal matrix 
consisting of the eigenvalues of $S$.
Each eigenvalue can be written as
\begin{equation}
\lambda_i = \left(\frac{l_{\min}}{N}\right)^2
\sum_{j=1}^3\min(i_j^2,(N-i_j)^2),
\label{eq:eig-Laplacian-spectral}
\end{equation}
where $i$ is a zero-based multi-index in an $N^3$ grid,
and $l_{\min}$ is the minimum sampling rate (minimum number of points per wavelength).
Clearly,
\[
\rho(S)= \max_i|\lambda_i|
\leq \left(\frac{l_{\min}}{N}\right)^2 \sum_{d=1}^3\left(\frac{N}{2}\right)^2 = \frac{3}{4}l_{\min}^2.
\]
The standard seven-point stencil can also be used to generate a sparse matrix $\tilde{S}$,
then the eigenvalues are replaced by
\[
\tilde{\lambda}_i = \left(\frac{l_{\min}}{2\pi}\right)^2\sum_{j=1}^3 2(1-\cos\frac{i_j}{N}\pi).
\]
The spectral radius is instead
$
\rho(\tilde{S})= \max_i|\tilde{\lambda}_i| \leq \frac{3}{\pi^2} l_{\min}^2.
$
Low-order finite difference methods may need a large sampling rate $l_{\min}$,
which increases the size and spectral radius of $A$.

The mass matrix $M$ contains the variations of the wavespeed.
For the simplest diagonal case, the $i$th non-zero diagonal entry is simply
\[
M_{ii} = \frac{l_{\min}^2}{l_i^2},
\]
where $l_i$ is the local sampling rate on the $i$th grid point. 
For this case,
$\rho_1$ defined in Lemma \ref{lem:spectrum-1x1} satisfies
\[
\rho_1=\rho(S-M+I) \leq \|S\|_2 + \|I-M\|_2 = O(l_{\min}^2) + \frac{l_{\max}^2-l_{\min}^2}{l_{\max}^2}.
\]
Regarding the spreading of the spectrum, 
the minimum sampling rate ($l_{\min}$) has a primary influence,
and the variations of the wavespeed ($l_{\max}/l_{\min}$) play a secondary role.

\subsection{Cost of shifted solution}
{
We first compare the two formulations discussed in Section \ref{sec:spectrum}.
We test the cost of solving a pair of shifted problems 
corresponding to the two cases \eqref{eq:lin-sys} and \eqref{eq:mat-2x2}.
In order to draw a fair comparison, we force the shifted problems to be equivalent.
Let $A$ be a Hermitian indefinite matrix,
and complex numbers $z$, $s$ satisfy $z+1 = (s+1)^2$. Then the pair of shifted problems are:
\begin{equation}
(A-zI)y = f,
\label{eq:shifted-1x1}
\end{equation}
and
\begin{equation}
\begin{pmatrix}
-(s+1)I & \mathrm{i}I\\
-\mathrm{i}(A + I) & -(s+1)I
\end{pmatrix}
\begin{pmatrix}
\mathrm{i}y \\ (s+1)y
\end{pmatrix}
=
\begin{pmatrix}
0 \\ f
\end{pmatrix}.
\label{eq:shifted-2x2}
\end{equation}
The results are tabulated in Table \ref{tab:shift-inv}. 
Note that the cost of each matrix-vector product is roughly the same.
\eqref{eq:shifted-1x1} is more suitable for
the case that the spectrum of $A$ is compact and the imaginary shift is large,
otherwise \eqref{eq:shifted-2x2} is better suited.
}

\begin{table}[ptbh]
\centering
\caption{Number of matrix-vector products for reducing the residual by $10^2$.
Boldface numbers indicate a superior performance over the alternative case.}
\label{tab:shift-inv}
\tabcolsep=1.6mm\renewcommand{\arraystretch}{1.1}
\begin{tabular}
[c]{|c||c|c|c|c|}
\multicolumn{5}{c}{(a) Case one: solving \eqref{eq:shifted-1x1}}\\
\hline
Spectrum & $z=\mathrm{i}$ & $z=\mathrm{i}/2$ & $z=\mathrm{i}/4$ & $z=\mathrm{i}/8$\\
\hline
$[-1,8]$  & {\bf 28} & {\bf 67} & {\bf 156} & {\bf 276}\\
\hline
$[-1,16]$ & {\bf 52} & {\bf 145} & 381 & 721\\
\hline
$[-1,32]$ & {\bf 97} & 436 & 861 & 1691\\
\hline
$[-1,64]$ & 331 & 916 & 1831 & 3736\\
\hline
\end{tabular}
\begin{tabular}
[c]{|c||c|c|c|c|}
\multicolumn{5}{c}{ }\\
\multicolumn{5}{c}{(b) Case two: solving \eqref{eq:shifted-2x2}}\\
\hline
Spectrum & $z=\mathrm{i}$ & $z=\mathrm{i}/2$ & $z=\mathrm{i}/4$ & $z=\mathrm{i}/8$\\
\hline
$[-1,8]$ & 67 & 121 & 209 & 441\\
\hline
$[-1,16]$ & 85 &  157 &  {\bf 229} &  {\bf 609}\\
\hline
$[-1,32]$ & 113 &  {\bf 193} & {\bf 397} & {\bf 681}\\
\hline
$[-1,64]$ & {\bf 149} & {\bf 221} & {\bf 449} & {\bf 909}\\
\hline
\end{tabular}
\end{table}

\subsection{Scaling test}
This scaling test checks the cost of solving \eqref{eq:lin-sys} 
as the frequency and the problem size increase.
We choose a wavespeed function that has eight high-wavespeed anomalies.
Figure \ref{fig:mod320} visualizes the wavespeed function.
Figure \ref{fig:sol320} shows the solution of the largest problem size.

\begin{figure}[ptbh]
\centering
\includegraphics[width=\textwidth]{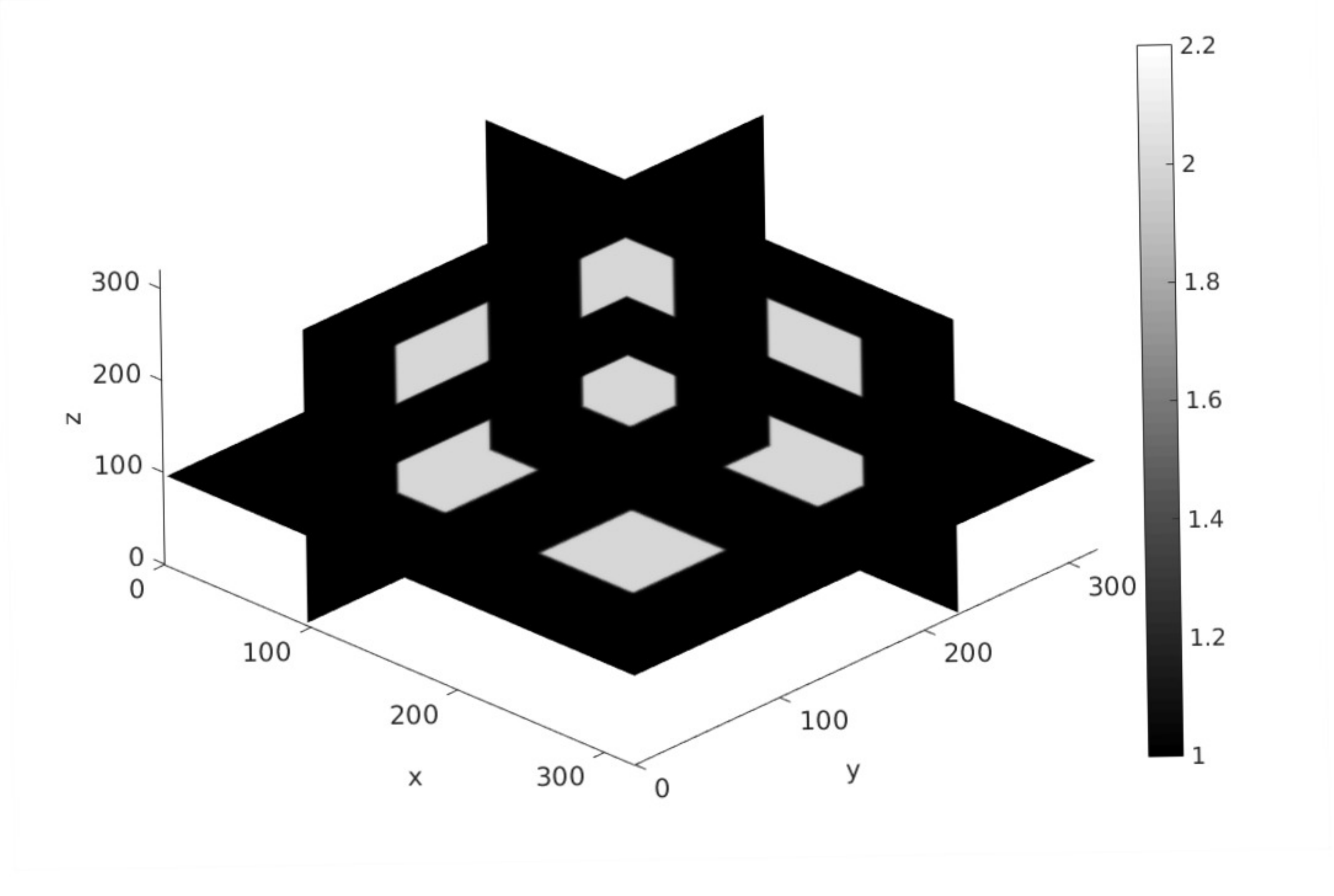}\\
\caption{Wavespeed function.}
\label{fig:mod320}
\end{figure}
\begin{figure}[ptbh]
\centering
\includegraphics[width=\textwidth]{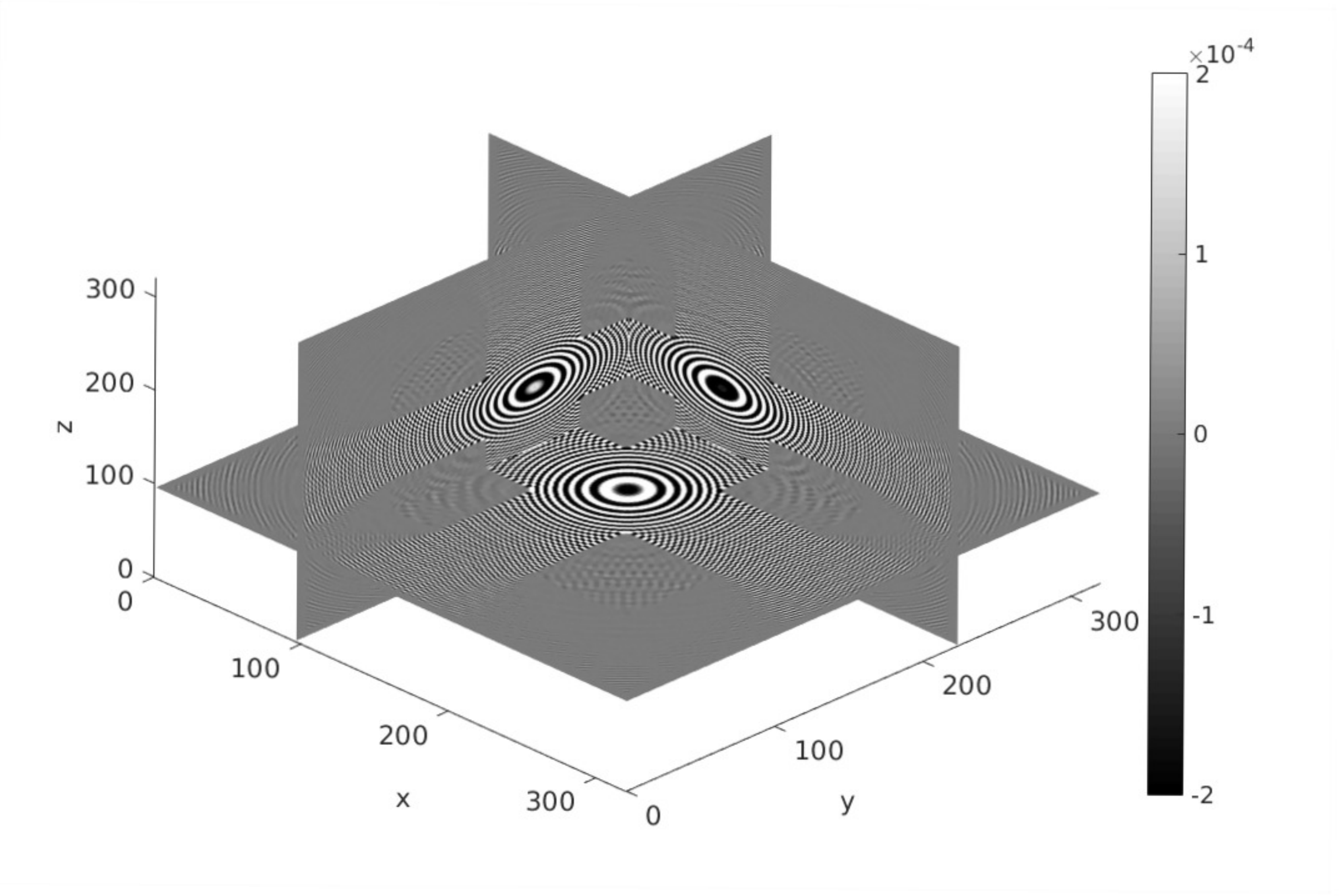}\\
\caption{Solution of the problem in Figure \ref{fig:mod320}.}
\label{fig:sol320}
\end{figure}

We set up the solver based on the techniques described in Section \ref{sec:tech}.
Six quadrature points are used, 
and their locations change with respect to the angular frequency $\omega$.
For example, for the four problems in Table \ref{tab:scaling}(a),
we fix $t=0.1$ in \eqref{eq:quad-points} 
and select $\epsilon\propto 1/\omega$ in \eqref{eq:quad-radius},
then the sets of quadrature points are
\[
\begin{pmatrix}
\phantom{-}0.00 + 0.80\mathrm{i}\\
-0.22 + 2.05\mathrm{i}\\
-0.43 + 0.80\mathrm{i}\\
-0.43 - 1.70\mathrm{i}\\
-0.22 - 2.96\mathrm{i}\\
\phantom{-}0.00 - 1.70\mathrm{i}\\
\end{pmatrix},\quad
\begin{pmatrix}
\phantom{-}0.00 + 0.40\mathrm{i}\\
-0.15 + 1.25\mathrm{i}\\
-0.30 + 0.40\mathrm{i}\\
-0.30 - 1.30\mathrm{i}\\
-0.15 - 2.16\mathrm{i}\\
\phantom{-}0.00-1.30\mathrm{i}\\
\end{pmatrix},\quad
\begin{pmatrix}
\phantom{-}0.00 + 0.27\mathrm{i}\\
-0.12 + 0.99\mathrm{i}\\
-0.25 + 0.27\mathrm{i}\\
-0.25 - 1.17\mathrm{i}\\
-0.12 - 1.89\mathrm{i}\\
\phantom{-}0.00 - 1.17\mathrm{i}\\
\end{pmatrix}
,\quad
\begin{pmatrix}
\phantom{-}0.00 + 0.20\mathrm{i}\\
-0.11 + 0.85\mathrm{i}\\
-0.23 + 0.20\mathrm{i}\\
-0.23 - 1.10\mathrm{i}\\
-0.11 - 1.76\mathrm{i}\\
\phantom{-}0.00 - 1.10\mathrm{i}\\
\end{pmatrix}.
\]
The first quadrature point in each set is closest to the origin. For these points, the parameters for applying \eqref{eq:fix-pt-exp}
are listed in Table \ref{tab:param}.

\begin{table}[ptbh]
\centering
\caption{Parameters of the polynomial iteration \eqref{eq:fix-pt-exp} for one pole.
The parameters are computed by solving \eqref{eq:fix-pt-conv}.
$z$ is the complex shift.
The method stops when the  residual is reduced by 5.
}
\label{tab:param}
\tabcolsep=1.6mm\renewcommand{\arraystretch}{1.1}
\begin{tabular}
[c]{|c|c|c|c|c|}
\hline
$n$     & $z$ & $q$ & $\delta$  & \textsf{mvs}\\
\hline
$80^3$  & 0.8$\mathrm{i}$ & 2 & 0.51 & 10\\
\hline
$160^3$ & 0.4$\mathrm{i}$ & 3 & 0.91 & 18\\
\hline
$240^3$ & 0.27$\mathrm{i}$ & 3 & 0.88 & 27\\
\hline
$320^3$ & 0.2$\mathrm{i}$ & 3 & 0.85 & 39\\
\hline
\end{tabular}
\end{table}

Regarding different cases in Section \ref{sec:spectrum},
by estimating the spectrum we find that for the spectral method one can apply 
Algorithm \ref{alg:fci} to the original matrix $A$ because the spectrum is more compact,
and the modified matrix $\mathrm{i}C - I$ in \eqref{eq:mat-2x2}
is suitable for the finite difference method.
The spectral method case includes a 
regularized inverse Laplacian preconditioner is diagonalized by FFT 
with eigenvalues $\{1/\max(\lambda_i,1)\}$,
where $\lambda_i$ is defined in \eqref{eq:eig-Laplacian-spectral};
the finite difference case includes a ILU(0) preconditioner
based on the 7-point stencil discrete Laplacian.

\begin{table}[ptbh]
\centering
\caption{Scaling test for fixed sampling rate and increasing problem sizes.}
\label{tab:scaling}
\tabcolsep=1.6mm\renewcommand{\arraystretch}{1.1}

\begin{tabular}
[c]{|c|c|c|c|c|}
\multicolumn{5}{c}{(a) Fourier spectral method}\\
\hline
$n$   & $\omega/(2\pi)$ & \textsf{its} & \textsf{mvs}  & \textsf{i-t}\\
\hline
$80^3$  & 35.56  & 6 & 879   &    39.8 \\
\hline
$160^3$ & 71.11  & 8 & 1795  &   719.6\\
\hline
$240^3$ & 106.67 & 9 & 2670  &  4610.2\\
\hline
$320^3$ & 142.22 & 11 & 3754 & 14841.6 \\
\hline
\multicolumn{5}{c}{ }\\
\multicolumn{5}{c}{(b) Finite difference method}\\
\hline
$n$   & $\omega/(2\pi)$ & \textsf{its} & \textsf{mvs}  & \textsf{i-t}\\
\hline
$80^3$  & 8.89  &  9 &  341 &  18.9\\
\hline
$160^3$ & 17.78 &  8 &  536 & 293.7\\
\hline
$240^3$ & 26.67 & 11 &  842 & 1649.2\\
\hline
$320^3$ & 35.56 & 10 &  1065 & 4954.5\\
\hline
\end{tabular}
\end{table}

\begin{figure}[ptbh]
\begin{center}%
\begin{tabular}
[c]{cc}%
\includegraphics[width=0.48\textwidth]{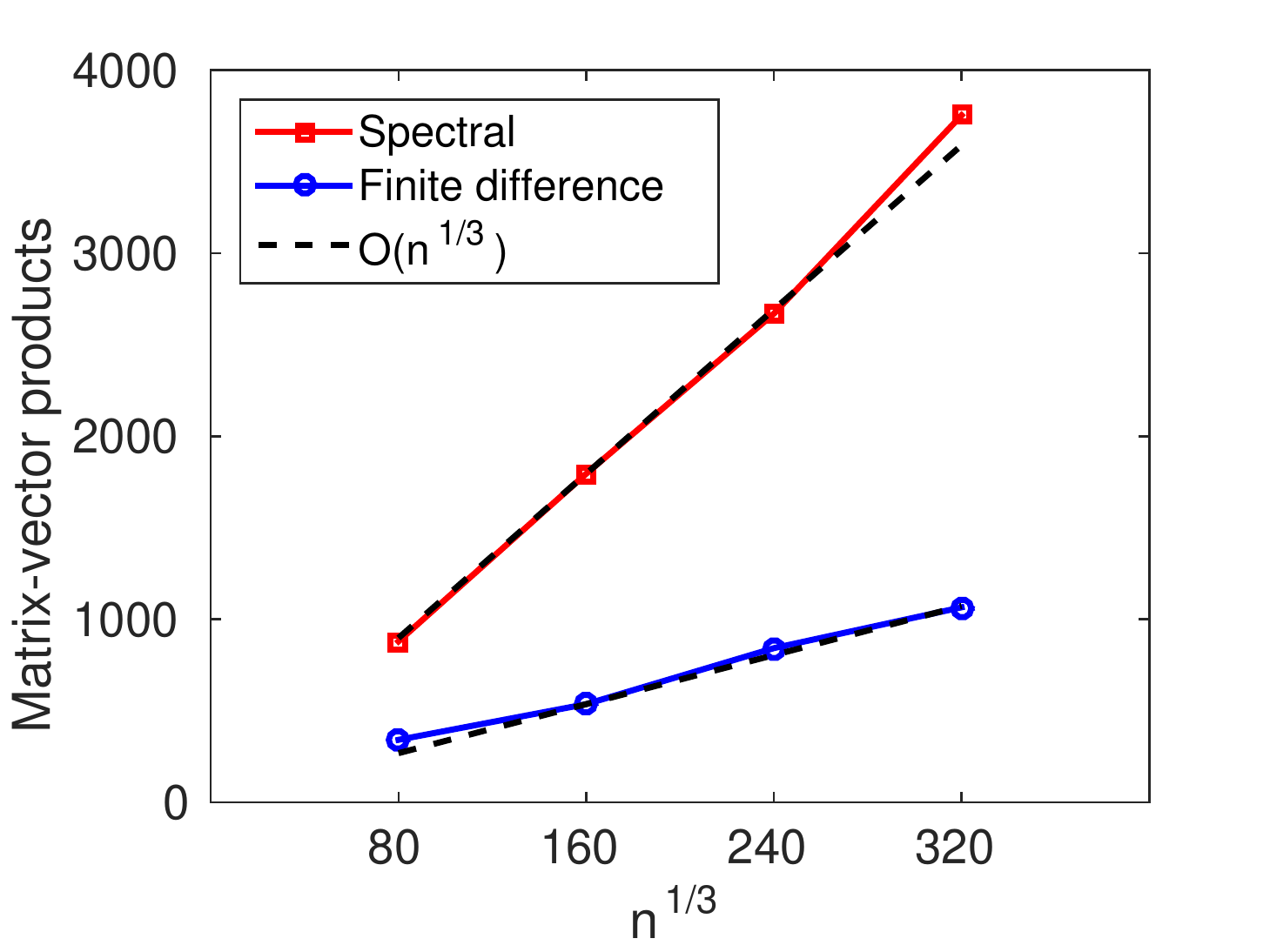} &
\includegraphics[width=0.48\textwidth]{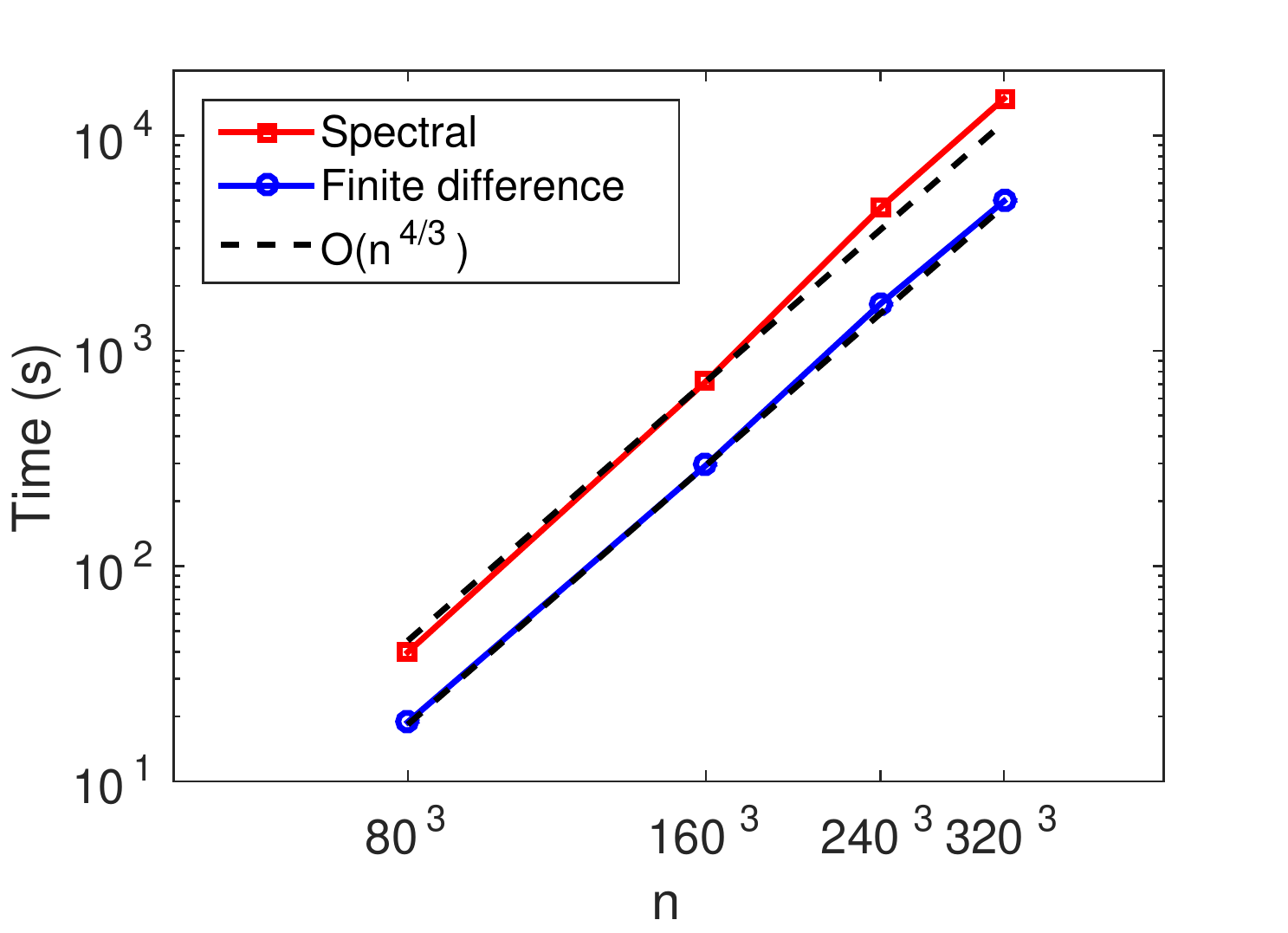}\\
{\small (a) \textit{Matrix-vector products}} & {\small (b)
\textit{Iteration time}}%
\end{tabular}
\end{center}
\caption{3D Scaling test.}%
\label{fig:scaling}%
\end{figure}

Table \ref{tab:scaling} and Figure \ref{fig:scaling} 
are the test results for reducing the residual by $10^{3}$.
For the spectral method, the sampling rate is $2.25$ in the background,
and is $4.5$ inside the anomalies.
For the finite difference approach, the frequency is reduced by four times
to obtain a minimum sampling rate of $9$.
For both cases, the number of matrix-vector products is proportional
to the angular frequency $O(\omega)$.

\section{SEG/EAGE salt-dome model}
The SEG/EAGE salt-dome model \cite{saltdome} is 
a 3D wavespeed model commonly used in exploration geophysics.
The physical size is 12km$\times$12km$\times$4.5km.
The wavespeed ranges between 1500m/s and 4500m/s,
see Figure \ref{fig:mod-salt} for sections of the wavespeed.
At high frequency -- 33.33Hz, we apply Fourier spectral discretization
on a 201$\times$676$\times$676 grid.
Figure \ref{fig:sol-salt} visualizes the solution wavefield.
For $10^{-2}$, $10^{-4}$, and $10^{-6}$ relative residual, the proposed method takes 
17 iterations (984min), 32 iterations (1850min), and 51 iterations (2943min), respectively.
Figure \ref{fig:res-salt} shows that linear convergence of the residual still holds
even when the wavespeed is rather complex.

\begin{figure}[ptbh]
\centering
\includegraphics[width=\textwidth]{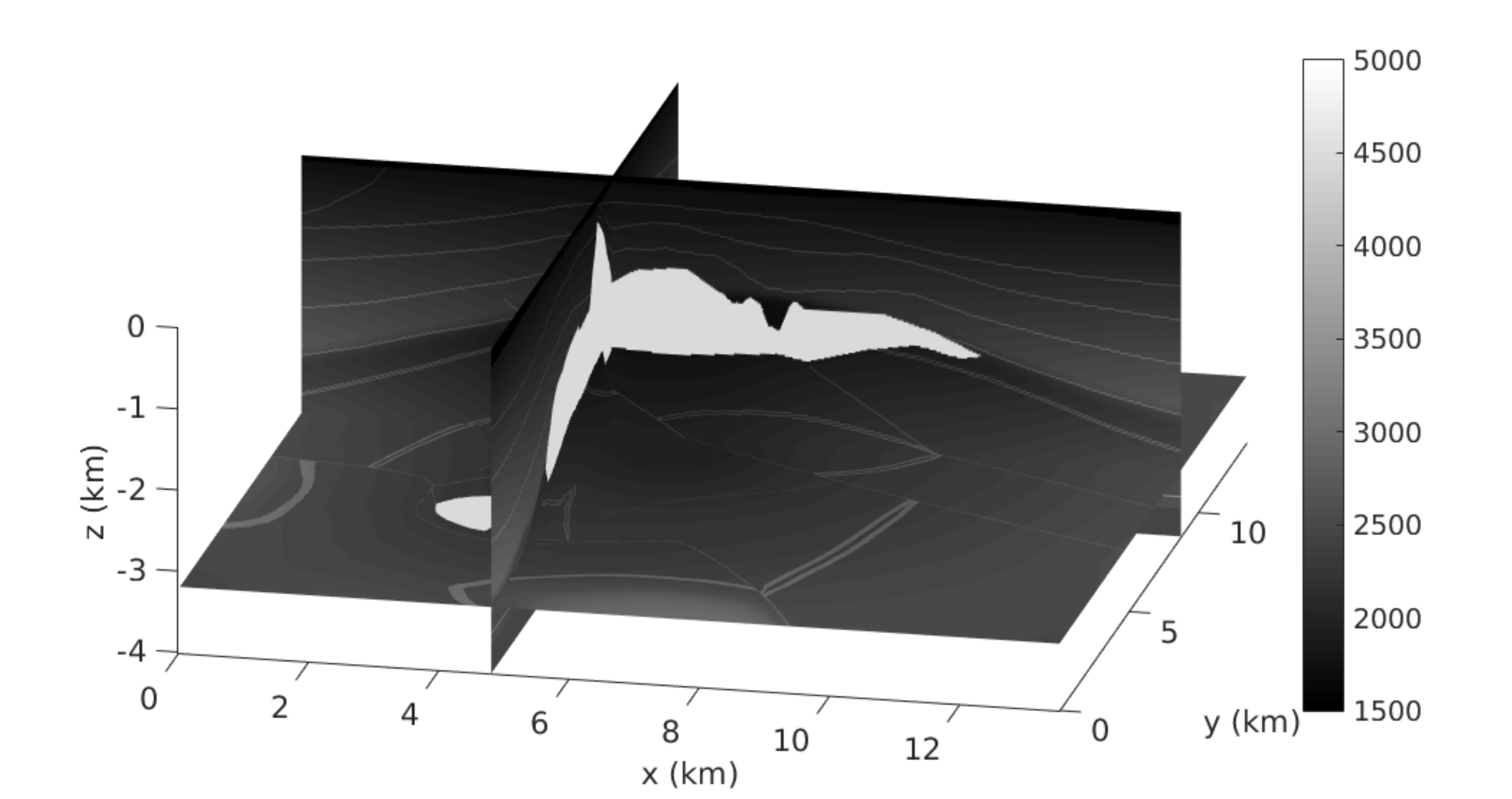}\\
\caption{Wavespeed function of the SEG/EAGE salt-dome model.}
\label{fig:mod-salt}
\end{figure}
\begin{figure}[ptbh]
\centering
\includegraphics[width=\textwidth]{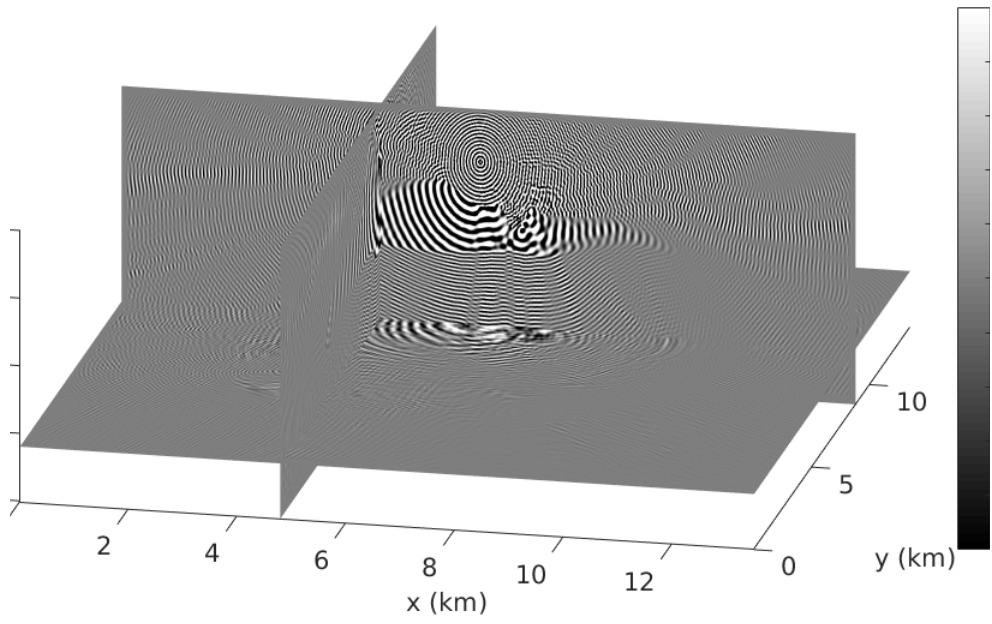}\\
\caption{33.33Hz solution wavefield corresponding to Figure \ref{fig:mod-salt}.}
\label{fig:sol-salt}
\end{figure}

\begin{figure}[ptbh]
\centering
\includegraphics[width=\textwidth]{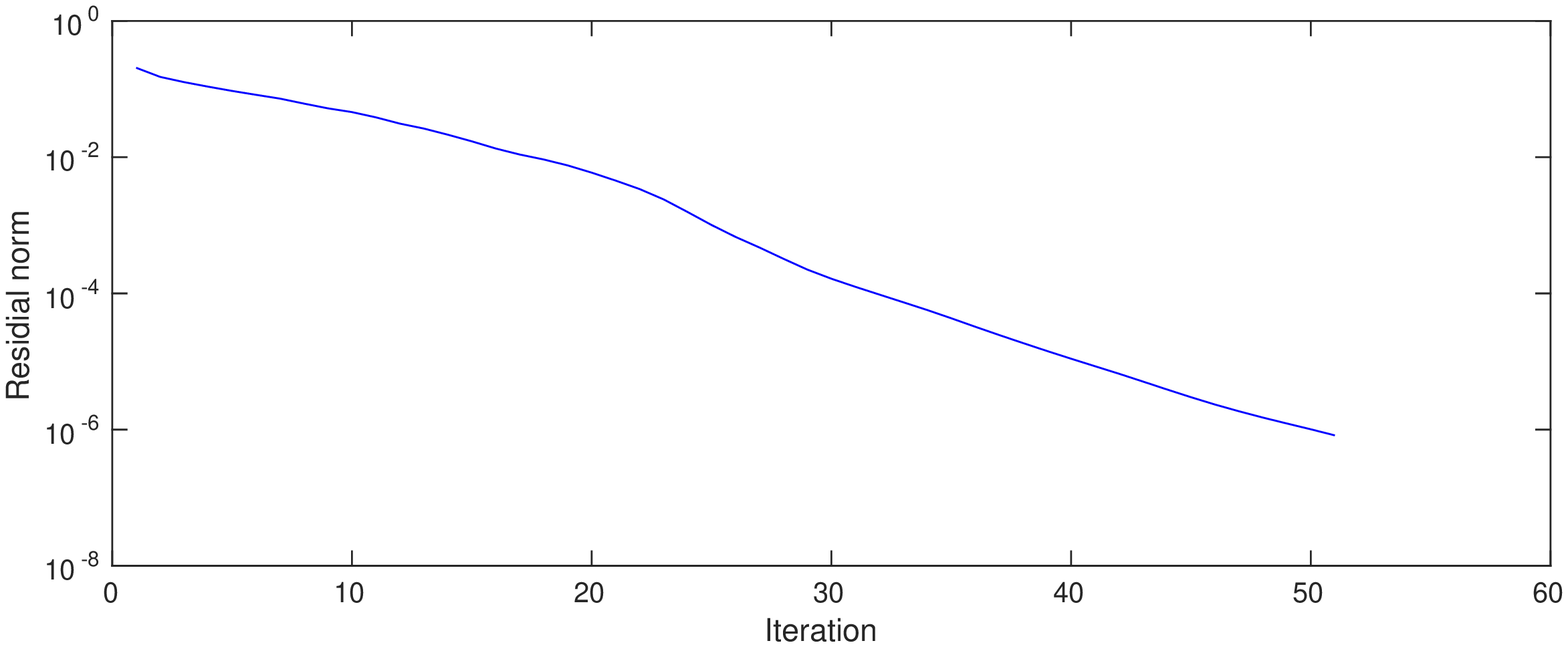}\\
\caption{Residual history of FCI iterations of SEG/EAGE salt-dome model.}
\label{fig:res-salt}
\end{figure}

\section{Conclusions}
An iterative method was proposed to solve the discretized 3D high-frequency Helmholtz equation.
In the framework of contour integration method
which implicitly decomposes the original problem into an inner and outer problem,
a fixed-point iteration was introduced to solve the outer problem.
GMRES is suitable for solving the inner problem because of
our theoretical estimates on the distribution of eigenvalues.
3D numerical examples show that the computational cost
of this method scales as $O(\omega^4)$ or $O(n^{4/3})$.
The method is especially suitable for solving high-frequency problems when combined with spectral methods.

\section*{Acknowledgements}
Maarten V. de Hoop gratefully acknowledges support from the Simons Foundation
under the MATH + X program, the NSF under grant DMS-1559587, and the corporate
members of the Geo-Mathematical Group at Rice University and Total. Yuanzhe Xi and Yousef Saad acknowledge support from NSF DMS-1521573 and the Minnesota Supercomputing Institute.


\begin{thebibliography}{99}
\bibitem {wrigley} \textsc{H. E. Wrigley},
\emph{Accelerating the Jacobi method for solving simultaneous equations by
Chebyshev extrapolation when the eigenvalues of the iteration matrix are complex},
Comput. J., 6 (1963), pp. 169--176.
\end{thebibliography}
\end{document}